\documentclass[12pt]{article}
\usepackage{fullpage,amsthm,amsmath,amssymb,xspace,verbatim,tikz,subfig,setspace,url}
\usetikzlibrary{positioning, shapes.geometric, calc}

\newcounter{thmctr}
\newtheorem{thm}[thmctr]{Theorem}
\newtheorem{lemma}[thmctr]{Lemma}
\newtheorem{prop}[thmctr]{Proposition}
\newtheorem{cor}[thmctr]{Corollary}

\newtheorem*{definition}{Definition}

\theoremstyle{definition}

\newtheorem*{remarks}{Remarks}

\theoremstyle{plain}

\ifx\buildsubmitted\undefined
    \newenvironment{notarxiv}{\comment}{\endcomment}
    \newenvironment{onlyarxiv}{}{}
\else
    \newenvironment{notarxiv}{}{}
    \newenvironment{onlyarxiv}{\comment}{\endcomment}
\fi
\tikzstyle{pathscale}=[]
\tikzstyle{vertex}=[circle,fill=black,inner sep=2pt]
\tikzstyle{vertrect}=[draw,rectangle,inner sep=2pt]
\tikzstyle{vertdia}=[draw,diamond,inner sep=2pt]
\newcommand{\stocnl}{}
\newcommand{\stocnla}{}
\newcommand{\prooftext}{Proof\xspace}

\newcommand{\dotp}[2]{\left< #1, #2 \right>}

\newcommand{\propeig}[1]{\texttt{Eig}[#1]\xspace}
\newcommand{\propexpand}[1]{\texttt{Expand}[#1]\xspace}
\newcommand{\propcount}[1]{\texttt{Count}[#1]\xspace}
\newcommand{\propcycle}[2][4]{\texttt{Cycle$_{#1}$}[#2]\xspace}

\newcommand{\dhruvuni}{University of Illinois at Chicago \\ mubayi@math.uic.edu}
\newcommand{\johnuni}{University of Illinois at Chicago \\ lenz@math.uic.edu}
\newcommand{\dhruvfoot}{\footnote{Research supported in part by  NSF Grants 0969092 and 1300138.}}
\newcommand{\johnfoot}{\footnote{Research partly supported by NSA Grant H98230-13-1-0224.}}

\title{Eigenvalues and Linear Quasirandom Hypergraphs}
\author{John Lenz \johnfoot \\ \johnuni \and Dhruv Mubayi \dhruvfoot \\ \dhruvuni}

\begin{document}
\maketitle

\renewcommand{\thefootnote}{\fnsymbol{footnote}} 
\renewcommand{\thefootnote}{\arabic{footnote}}

\begin{abstract}
  Let $p(k)$ denote the partition function of $k$.  For each $k \ge 2$,  we describe a list of $p(k)
  -1$ quasirandom properties that a $k$-uniform hypergraph can have.  Our work connects previous
  notions on linear hypergraph quasirandomness of Kohayakawa-R\"odl-Skokan and
  Conlon-H\`{a}n-Person-Schacht and the spectral approach of Friedman-Wigderson.  For each of the
  quasirandom properties that are described, we define a largest and second largest eigenvalue.  We
  show that a hypergraph satisfies these quasirandom properties if and only if it has a large
  spectral gap.  This answers a question of Conlon-H\`{a}n-Person-Schacht.  Our work can be viewed
  as a partial extension to hypergraphs of the seminal spectral results of Chung-Graham-Wilson for
  graphs.
\end{abstract}

\section{Introduction} 
\label{sec:Introduction}

The study of quasirandom or pseudorandom graphs was initiated by Thomason~\cite{qsi-thomason87,
qsi-thomason87-2} and then refined by Chung, Graham, and Wilson~\cite{qsi-chung89}, resulting in a
list of equivalent (deterministic) properties of graph sequences which are inspired by $G(n,p)$.
Beginning with these foundational papers on the
subject~\cite{qsi-chung89,qsi-thomason87,qsi-thomason87-2}, the last two decades have seen an
explosive growth in the study of quasirandom structures in mathematics and computer science.  For
details on quasirandomness, we refer the reader to a survey of Krivelevich and
Sudakov~\cite{qsi-survey-krivelevich06} for graphs and recent papers of
Gowers~\cite{hqsi-gowers06,rrl-gowers07,qsi-gowers08} for general quasirandom structures including
hypergraphs.

\subsection{Previous Results} 
\label{sub:introprev}

The core of what Chung, Graham, and Wilson~\cite{qsi-chung89} proved is that several properties of
graph sequences are equivalent. Two of them are \texttt{Disc} and \texttt{Count[All]}.  The first
states that all sufficiently large vertex sets have the same edge density as the original graph and
the second states that for all fixed graphs $F$ the number of copies of $F$ is what one would expect
in a random graph with the same density.

A \emph{$k$-uniform hypergraph} is a pair of finite sets $(V(H),E(H))$ where $E(H) \subseteq
\binom{V(H)}{k}$ is a collection of $k$-subsets of $V(H)$.  For $U \subseteq V(H)$, the induced
subgraph on $U$, denoted $H[U]$, is the hypergraph with vertex set $U$ and edge set $\{ e \in E(H) :
e \subseteq U \}$.  If $F$ and $G$ are hypergraphs, a \emph{labeled copy of $F$ in $H$} is an
edge-preserving injection $V(F) \rightarrow V(H)$, i.e.\ an injection $\alpha : V(F) \rightarrow
V(H)$ such that if $E$ is an edge of $F$, then $\{ \alpha(x) : x \in E \}$ is an edge of $H$.  A
\emph{graph} is a $2$-uniform hypergraph.

Almost immediately after proving their theorem, Chung and Graham~\cite{hqsi-chung12, hqsi-chung90,
hqsi-chung90-2, hqsi-chung91, hqsi-chung92} investigated generalizing the theorem to $k$-uniform
hypergraphs. One initial difficulty in generalizing quasirandomness to $k > 2$ is an observation by
R\"odl that a construction of Erd\H{o}s and Hajnal~\cite{ram-erdos72} shows that the hypergraph
generalizations of \texttt{Disc} and \texttt{Count[All]} are not equivalent.  Motivated by this,
Chung and Graham~\cite{hqsi-chung90, hqsi-chung90-2, hqsi-chung91, hqsi-chung92} investigated how to
strengthen the property \texttt{Disc} to make it equivalent to \texttt{Count[All]}.  They found
several properties equivalent to \texttt{Count[All]}; the main property they use is related to the
number of even/odd subgraphs of a given hypergraph which they called \texttt{Deviation}.
Simultaneously, Frankl and R\"odl~\cite{rrl-frankl92} also obtained a property stronger than
\texttt{Disc} which is equivalent to \texttt{Count[All]}.  Subsequently, other properties equivalent
to \texttt{Count[All]} have been studied by several
researchers~\cite{hqsi-austin10,hqsi-gowers06,hqsi-keevash09,hqsi-kohayakawa02}.

It remained open whether the simpler property \texttt{Disc} for $k$-uniform hypergraphs is
equivalent to counting some class of hypergraphs or counting a single substructure.  This is related
to the Weak Hypergraph Regularity Lemma~\cite{rrl-chung91, rrl-frankl92, rrl-steger90}.  Recently,
Kohayakawa, Nagle, R\"odl, and Schacht~\cite{hqsi-kohayakawa10} answered this question by showing
that \texttt{Disc} is equivalent to counting the family of linear hypergraphs, where a hypergraph
$H$ is \emph{linear} if every pair of distinct edges share at most one vertex.  Building on this,
Conlon, H\`{a}n, Person, and Schacht~\cite{hqsi-conlon12} showed that \texttt{Disc} is equivalent to
counting a type of linear four cycle.   These two results can be combined into the following
theorem.

\begin{thm} \label{thm:conlonhyperquasibottom}
  (\textbf{Kohayakawa-Nagle-R\"odl-Schacht~\cite{hqsi-kohayakawa10} and 
  Conlon-H\`{a}n-Person-Schacht~\cite{hqsi-conlon12}}) Let $0 < p < 1$ be a fixed constant and let
  $\mathcal{H} = \{H_n\}_{n\rightarrow\infty}$ be a sequence of $k$-uniform hypergraphs such that
  $|V(H_n)| = n$ and $|E(H_n)| \geq p \binom{n}{k} + o(n^k)$.  The following properties are
  equivalent:
  \begin{itemize}
    \item \texttt{Disc}: For every $U \subseteq V(H_n)$, $|E(H_n[U])| = p \binom{|U|}{k} + o(n^k)$.

    \item \texttt{Count[linear]}: For every fixed linear $k$-uniform hypergraph $F$ with $e$ edges
      and $f$ vertices, the number of labeled copies of $F$ in $H_n$ is $p^e n^f + o(n^f)$.

    \item \texttt{Cycle$_4$}: The number of labeled copies of $C_4$ in $H_n$ is at most
      $p^{|E(C_4)|} n^{|V(C_4)|} + o(n^{|V(C_4)|})$, where $C_4$ is a linear hypergraph defined precisely
      in Section~\ref{sub:cycledefinion}.
  \end{itemize}
\end{thm}
 
Note that Conlon et al.~\cite{hqsi-conlon12} put the condition ``$|E(H_n)| \geq p \binom{n}{k} +
o(n^k)$'' into the statement of the properties that don't trivially imply it like \texttt{Disc} and
this is equivalent to the way we have stated Theorem~\ref{thm:conlonhyperquasibottom}.  Conlon et
al.~\cite{hqsi-conlon12} have several more properties including induced subgraph counts and common
neighborhood sizes, but we consider the properties stated in
Theorem~\ref{thm:conlonhyperquasibottom} as the core properties.

\subsection{Our Results} 
\label{sub:introour}

Another graph property equivalent to \texttt{Disc} is \texttt{Eig}, which states that if $\mu_1$ and
$\mu_2$ are the first and second largest (in absolute value) eigenvalues of the adjacency matrix of
the graph respectively, then $\mu_2 = o(\mu_1)$.  Neither Chung and Graham~\cite{hqsi-chung90,
hqsi-chung90-2, hqsi-chung91, hqsi-chung92} nor Kohayakawa, R\"odl, and
Skokan~\cite{hqsi-kohayakawa02} provided a generalization of \texttt{Eig} to hypergraphs. Later,
Conlon, H\`{a}n, Person, and Schacht~\cite{hqsi-conlon12} asked whether there exists a
generalization of \texttt{Eig} to $k$-uniform hypergraphs which is equivalent to \texttt{Disc}. The
eigenvalue description of graph quasirandomness  has proved to be a very useful result to show that
certain explicitly constructed graphs are quasirandom
(see~\cite{ee-alon01,mr-alon05,mr-lenz11,ee-szabo03}).

This leads to our first main contribution. We define a generalization of \texttt{Eig} to $k$-uniform
hypergraphs and add it into the equivalences stated in Theorem~\ref{thm:conlonhyperquasibottom}.
This answers the aforementioned question of Conlon et al.~\cite{hqsi-conlon12}.

Our second contribution is to generalize Theorem~\ref{thm:conlonhyperquasibottom} to a slightly
larger class of hypergraphs. Let $k \geq 2$ be an integer and let $\pi$ be a proper partition of
$k$, by which we mean that $\pi$ is an unordered list of at least two positive integers whose sum is
$k$.  For the partition $\pi$ of $k$ given by $k = k_1 + \dots + k_t$, we will abuse notation by
saying that $\pi = k_1 + \dots + k_t$.  For every proper partition $\pi$, we define properties
\propexpand{$\pi$}, \propeig{$\pi$}, and \propcycle{$\pi$} and show that they are equivalent.

\begin{definition}
  Let $k \geq 2$ and let $\pi = k_1+\dots+k_t$ be a proper partition of $k$.  A $k$-uniform
  hypergraph $F$ is \emph{$\pi$-linear} if there exists an ordering $E_1,\dots,E_m$ of the edges of
  $F$ such that for every $i$, there exists a partition of the vertices of $E_i$ into
  $A_{i,1},\dots,A_{i,t}$ such that for $1 \leq s \leq t$, $|A_{i,s}|=k_s$ and for every $j < i$, there exists
  an $s$ such that $E_j \cap E_i \subseteq A_{i,s}$. 
\end{definition}

Our hypergraph eigenvalues are based on definitions of Friedman and
Wigderson~\cite{ee-friedman95,ee-friedman95-2} (see Section~\ref{sec:hypergrapheigenvalues}).  In
graphs, it is easier to study the eigenvalues of regular graphs (possibly with loops).  A similar
situation occurs for hypergraphs, so Friedman and Wigderson~\cite{ee-friedman95,ee-friedman95-2}
focused almost exclusively on the following notion of regular for hypergraphs.

\begin{definition}
  A \emph{$k$-uniform hypergraph with loops} $H$ consists of a finite set $V(H)$ and a collection
  $E(H)$ of $k$-element multisets of elements from $V(H)$.  Informally, every edge has size exactly
  $k$ but a vertex is allowed to be repeated inside of an edge.  A $k$-uniform hypergraph with loops
  $H$ is \emph{$d$-coregular} if for every $(k-1)$-multiset $S$, there are exactly $d$ edges which
  contain $S$.
\end{definition}

The following is our main theorem.

\begin{thm} \label{thm:mainequiv} \textbf{(Main Result)}
  Let $0 < p < 1$ be a fixed constant and let $\mathcal{H} = \{H_n\}_{n\rightarrow\infty}$ be a
  sequence of $k$-uniform hypergraphs with loops such that $|V(H_n)| = n$ and $H_n$ is $\left\lfloor
  pn \right\rfloor$-coregular.  Let $\pi = k_1 + \dots + k_t$ be a proper partition of $k$.  The
  following properties are equivalent:

  \begin{itemize}
    \item \propeig{$\pi$}: $\lambda_{1,\pi}(H_n) = pn^{k/2} + o(n^{k/2})$ and $\lambda_{2,\pi}(H_n)
      = o(n^{k/2})$, where $\lambda_{1,\pi}(H_n)$ and $\lambda_{2,\pi}(H_n)$ are the first and
      second largest eigenvalues of $H_n$ with respect to $\pi$, which are defined
      in Section~\ref{sec:hypergrapheigenvalues}.

    \item \propexpand{$\pi$}: For all $S_i \subseteq \binom{V(H_n)}{k_i}$ where $1 \leq i \leq t$,
      \begin{align*} e(S_1,\dots,S_t) =  p \prod_{i=1}^t \left| S_i \right| + o\left(n^{k}\right)
      \end{align*} where $e(S_1,\dots,S_t)$ is the number of tuples $(s_1,\dots,s_t)$ such
      that $s_1 \cup \dots \cup s_t$ is a hyperedge and $s_i \in S_i$.

    \item \propcount{$\pi$-linear}: If $F$ is an $f$-vertex, $m$-edge, $k$-uniform, $\pi$-linear
      hypergraph, then the number of labeled copies of $F$ in $H_n$ is $p^m n^f + o(n^f)$.

    \item \propcycle{$\pi$}: The number of labeled copies of $C_{\pi,4}$ in $H_n$ is at most
      $p^{|E(C_{\pi,4})|} n^{|V(C_{\pi,4})|} + o(n^{|V(C_{\pi,4})|})$, where $C_{\pi,4}$ is the
      hypergraph four cycle of type $\pi$ which is defined in Section~\ref{sub:cycledefinion}.

    \item \propcycle[4\ell]{$\pi$}:  the number of labeled copies of $C_{\pi,4\ell}$ in $H_n$ is at
      most $p^{|E(C_{\pi,4\ell})|} n^{|V(C_{\pi,4\ell})|} + \linebreak[1]
      o(n^{|V(C_{\pi,4\ell})|})$, where $C_{\pi,4\ell}$ is the hypergraph cycle of type $\pi$ and
      length $4\ell$ defined in Section~\ref{sub:cycledefinion}.
  \end{itemize}

  In fact, all implications above except \propcycle[4\ell]{$\pi$} $\Rightarrow$ \propeig{$\pi$} are
  true with the coregular condition replaced by the weaker condition that $|E(H_n)| \geq
  p\binom{n}{k} + o(n^k)$.
\end{thm}

\begin{remarks} \hspace{1cm}

  \begin{itemize}
    \item In a companion paper~\cite{hqsi-lenz-quasi12-nonregular}, we prove that
      \propcycle[4\ell]{$\pi$} $\Rightarrow$ \propeig{$\pi$} for all sequences $\mathcal{H} =
      \{H_n\}_{n\rightarrow\infty}$ where $H_n$ is a $k$-uniform hypergraph with loops, $|V(H_n)| =
      n$, and $|E(H_n)| \geq p \binom{n}{k} + o(n^k)$.

    \item Following Chung, Graham, and Wilson~\cite{qsi-chung89}, our results extend to sequences
      which are not defined for every $n$ as follows. Let $\mathcal{H} = \{H_{n_q}\}_{q \rightarrow
      \infty}$ be a sequence of hypergraphs such that $|V(H_{n_q})| = n_q$, $n_{q} < n_{q+1}$, and
      $|E(H_{n_q})| \geq p \binom{n_q}{k} + o(n_q^k)$, where now the little-$o$ expression means
      there exists a function $f(q)$ such that $|E(H_{n_q})| \geq p \binom{n_q}{k} + f(q)$ with
      $\lim_{q \rightarrow\infty} f(q) n_q^{-k} = 0$.  Similarly, when we say that property $P$
      (which might include a little-$o$ expression) implies a property $P'$, what we mean is that
      there exist functions $f(q)$ and $f'(q)$ such that $P(f(q))$ implies $P'(f'(q))$, where the
      notation $P(f(q))$ stands for the property $P$ with the little-$o$ replaced by the function
      $f(q)$.  

    \item If $\pi = 1+ \dots + 1$, the partition of $k$ into $k$ ones, then the equivalences
      \propexpand{$\pi$} $\Leftrightarrow$ \propcount{$\pi$-linear} $\Leftrightarrow$
      \propcycle{$\pi$} of Theorem~\ref{thm:mainequiv} constitute
      Theorem~\ref{thm:conlonhyperquasibottom}.  Therefore, the property \propeig{$1+\cdots+1$} is
      the spectral property that is equivalent to the weak quasirandom properties studied by
      Kohayakawa, R\"odl, and Skokan~\cite{hqsi-kohayakawa02} and Conlon, H\`{a}n, Person, and
      Schacht~\cite{hqsi-conlon12}.

    \item If $\pi'$ is a refinement of $\pi$, then clearly \propcount{$\pi$-linear} $\Rightarrow$
      \propcount{$\pi'$-linear} and so if $\{H_n\}_{n\rightarrow\infty}$ is a sequence satisfying
      the properties in Theorem~\ref{thm:mainequiv} for $\pi$, it satisfies the properties for
      $\pi'$.  In a companion paper~\cite{hqsi-lenz-poset12}, we show the converse: if $\pi'$ is not
      a refinement of $\pi$ then \propexpand{$\pi$} $\not\Rightarrow$ \propexpand{$\pi'$} so the
      property \propexpand{$\pi$} is distinct for distinct $\pi$ and arranged in a poset via
      partition refinement.

  \end{itemize}
\end{remarks}

The remainder of this paper is organized as follows.  In Section~\ref{sub:cycledefinion}, we define
the hypergraph cycles $C_{\pi,4}$.  Section~\ref{sec:hypergrapheigenvalues} gives the formal
definition of eigenvalues with respect to $\pi$.  Theorem~\ref{thm:mainequiv} is proved by showing a
chain of implications in the order stated in the theorem; Section~\ref{sec:eigen2expander} proves
\propeig{$\pi$} $\Rightarrow$ \propexpand{$\pi$}, Section~\ref{sec:expander2c4} proves
\propexpand{$\pi$} $\Rightarrow$ \propcount{$\pi$-linear}, and Section~\ref{sec:c4toeigenvalue}
shows that \propcycle[4\ell]{$\pi$} $\Rightarrow$ \propeig{$\pi$} for $d$-coregular hypergraphs with
loops. Throughout this paper, we use the notation $[n] = \{1,\dots,n\}$.

\section{Hypergraph Cycles}
\label{sub:cycledefinion}

In this section, we define the hypergraph cycles $C_{\pi,2\ell}$.
The hypergraph cycles $C_{\pi,2\ell}$ are defined by first defining steps, then defining a path as a
combination of steps, and finally defining the cycle as a path with its endpoints identified.

\begin{definition}
  Let $\vec{\pi} = (1,\dots,1)$ be the ordered partition of $t$ into $t$ parts. Define the
  \emph{step of type $\vec{\pi}$}, denoted $S_{\vec{\pi}}$, as follows.  Let $A$ be a vertex set of
  size $2^{t-1}$ where elements are labeled by binary strings of length $t-1$ and let
  $B_{2},\dots,B_t$ be disjoint sets of size $2^{t-2}$ where elements are labeled by binary strings
  of length $t-2$.  The vertex set of $S_{\vec{\pi}}$ is the disjoint union $A \dot\cup B_{2}
  \dot\cup \cdots \dot\cup B_t$. Make $\{a,b_{2},\dots,b_t\}$ a hyperedge of $S_{\vec{\pi}}$ if $a
  \in A$, $b_j \in B_j$, and the code for $b_{j+1}$ is equal to the code formed by removing the
  $j$th bit of the code for $a$.

  For a general $\vec{\pi} = (k_1,\dots,k_t)$, start with $S_{(1,\dots,1)}$ and enlarge each vertex
  into the appropriate size; that is, a vertex in $A$ is expanded into $k_1$ vertices and each
  vertex in $B_j$ is expanded into $k_j$ vertices.  More precisely, the vertex set of
  $S_{\vec{\pi}}$ is $(A \times [k_1]) \dot\cup (B_2 \times [k_2]) \dot\cup \cdots \dot\cup (B_t
  \times [k_t])$, and if $\{a,b_2,\dots,b_t\}$ is an edge of $S_{(1,\dots,1)}$, then
  $\{(a,1),\dots,(a,k_1),(b_2,1),\dots, \linebreak[1] (b_2,k_2),\dots, \linebreak[1]
  (b_t,1),\dots,(b_t,k_t)\}$ is a hyperedge of $S_{\vec{\pi}}$.

  This defines the \emph{step of type $\vec{\pi}$}, denoted $S_{\vec{\pi}}$.  Let $A^{(0)}$ be the
  ordered tuple of vertices of $A$ in $S_{\vec{\pi}}$ whose binary code ends with zero and $A^{(1)}$
  the ordered tuple of vertices of $A$ whose binary code ends with one, where vertices are listed in
  lexicographic order within each $A^{(i)}$.  These tuples $A^{(0)}$ and $A^{(1)}$ are the two
  \emph{attach tuples of $S_{\vec{\pi}}$}
\end{definition}

\begin{figure}[ht] 
\begin{center}
\subfloat[$S_{(1,1)}$]{%
\begin{tikzpicture}
  \node (x1) at (0,0) [vertex,label=below:0] {};
  \node (x2) at (0,1) [vertex,label=above:1] {};
  \node (z)  at (1.5,0.5) [vertex,label=below:$\emptyset$] {};

  \node (A) at (0,-1.3) {$A$};
  \node (B) at (1.5,-1.3) {$B_2$};

  \draw (x1) -- (z) -- (x2);
\end{tikzpicture}
} \quad \quad \quad \quad
\subfloat[$S_{(3,2)}$]{%
\begin{tikzpicture}
  \node (x1) at (0,0) [vertex] {};
  \node (x2) at (-0.3,-0.3) [vertex] {};
  \node (x3) at (0.3,-0.3) [vertex] {};

  \node (x4) at (0,1.5) [vertex] {};
  \node (x5) at (-0.3,1.8) [vertex] {};
  \node (x6) at (0.3,1.8) [vertex] {};

  \node (z1) at (2,.7) [vertex] {};
  \node (z2) at (2.4,.7) [vertex] {};

  \node (A) at (0,-1.3) {$A$};
  \node (B) at (2.2,-1.3) {$B_2$};

  \draw (1,1.2) circle[x radius=0.6cm, y radius=1.8cm, rotate=68];
  \draw (1,.25) circle[x radius=0.6cm, y radius=1.8cm, rotate=-68];
\end{tikzpicture}
}
\caption{Steps with $t = 2$}
\label{fig:stepswithtwoparts}
\end{center}
\end{figure}
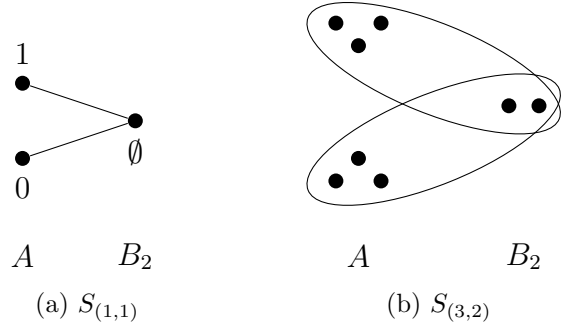 

Figure~\ref{fig:stepswithtwoparts} shows the steps of type $(1,1)$ and type $(3,2)$.
Notice that each step has ``length'' two if we consider the attach tuples as the ``ends'' of a path.

\begin{figure}[ht] 
\begin{center}
\subfloat[$S_{(1,1,1)}$]{%
\begin{tikzpicture}
  \node (x4) at (0,2) [vertex,label=left:11] {};
  \node (x3) at (0,1) [vertex,label=left:01] {};
  \node (x2) at (0,0) [vertex,label=left:10] {};
  \node (x1) at (0,-1) [vertex,label=left:00] {};

  \node (y1) at (1.3,0) [vertex,label=below:0] {};
  \node (y2) at (1.3,1) [vertex,label=above:1] {};
  \node (z1) at (2.6,0) [vertex,label=below:$0$] {};
  \node (z2) at (2.6,1) [vertex,label=above:$1$] {};

  \node (A1) at (0,-2.2) {$A$};
  \node (B2) at (1.3,-2.2) {$B_2$};
  \node (B3) at (2.6,-2.2) {$B_3$};

  \draw[rounded corners=12pt] (x1) -- (1.3,0) -- (z1);
  \draw[rounded corners=12pt] (x3) -- (1.3,1) -- (z1);
  \draw[rounded corners=12pt] (x2) -- (1.3,0) -- (z2);
  \draw[rounded corners=12pt] (x4) -- (1.3,1) -- (z2);
\end{tikzpicture}
} \quad \quad \quad
\subfloat[$S_{(1,1,1)}$]{%
\begin{tikzpicture}
  \node (x1) at (0,0) [vertex,label=left:00] {};
  \node (x2) at (0,1) [vertex,label=left:10] {};
  \node (y1) at (1.3,0.5) [vertex,label=below:0] {};
  \node (z1) at (2.6,0) [vertex,label=below:0] {};
  \node (z2) at (2.6,1) [vertex,label=above:1] {};
  \node (y2) at (3.9,0.5) [vertex,label=below:1] {};
  \node (x3) at (5.2,0) [vertex,label=right:01] {};
  \node (x4) at (5.2,1) [vertex,label=right:11] {};

  \node (A1) at (0,-1.3) {$A^{(0)}$};
  \node (B2) at (1.3,-1.3) {$B_2$};
  \node (B3) at (2.6,-1.3) {$B_3$};
  \node (B22) at (3.9,-1.3) {$B_2$};
  \node (A12) at (5.2,-1.3) {$A^{(1)}$};

  \draw[rounded corners=12pt] (x1) -- (1.3,0.5) -- (z1);
  \draw[rounded corners=12pt] (x2) -- (1.3,0.5) -- (z2);
  \draw[rounded corners=12pt] (x3) -- (3.9,0.5) -- (z1);
  \draw[rounded corners=12pt] (x4) -- (3.9,0.5) -- (z2);
\end{tikzpicture}
}
\caption{Steps of type $\pi = (1,1,1)$}
\label{fig:partialsteps}
\end{center}
\end{figure}
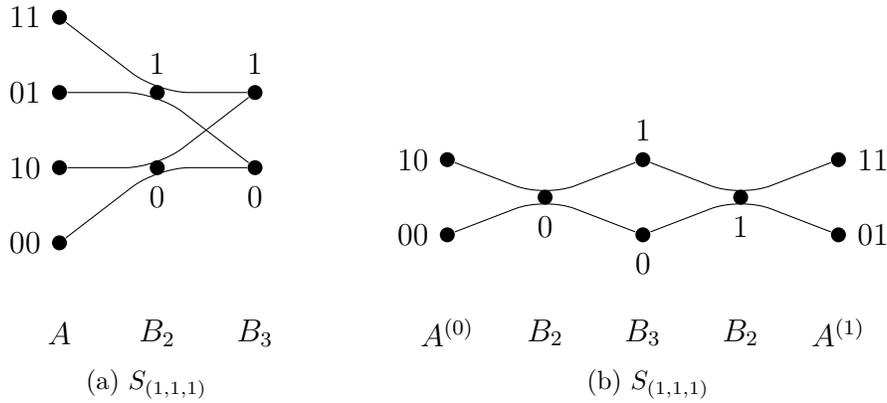 

Figure~\ref{fig:partialsteps} shows two different drawings of the step of type $\vec{\pi} = (1,1,1)$.
Notice that the attach tuples are easily visible in Figure~\ref{fig:partialsteps} \textit{(b)},
since the two attach tuples are the codes in $A$ ending with a zero and a one.  The step of
type $\vec{\pi} = (k_1,k_2,k_3)$ is an enlarged version of Figure~\ref{fig:partialsteps} similar to
Figure~\ref{fig:stepswithtwoparts} \textit{(b)}.

In general for arbitrary $\vec{\pi}$, the step $S_{\vec{\pi}}$ can be drawn in two ways similar to
Figure~\ref{fig:partialsteps}.  First from the definition, a step is a $k$-partite hypergraph with
parts $A,B_2,\dots,B_t$ so it can be drawn similar to Figure~\ref{fig:partialsteps} \textit{(a)}.
But the step can also be drawn with the two attach tuples on separate ends of the picture like
Figure~\ref{fig:partialsteps} \textit{(b)}.  Let $M_0$ be the set of edges incident to vertices in
the attach tuple $A^{(0)}$ and $M_1$ the set of edges incident to vertices in $A^{(1)}$.  Edges from
$M_0$ and $M_1$ intersect only in vertices in $B_t$ because if $a_0 \in A^{(0)}$ and $a_1 \in
A^{(1)}$ then the code for $a_0$ ends in a zero and the code for $a_1$ ends in a one, so only when
deleting the last bit will the codes possibly be the same.  Therefore, the step $S_{\vec{\pi}}$ can
be viewed as a type of length two path in a hypergraph formed from a collection of $k$-partite edges
$M_0$ between $A^{(0)}$ and $B_t$ and another collection of $k$-partite edges $M_1$ between $B_t$
and $A^{(1)}$.

\begin{definition}
  Let $\ell \geq 1$.  The \emph{path of type $\vec{\pi}$ of length $2\ell$}, denoted
  $P_{\vec{\pi},2\ell}$, is the hypergraph formed from $\ell$ copies of $S_{\vec{\pi}}$ with
  successive attach tuples identified.  That is, let $T_1,\dots,T_{\ell}$ be copies of
  $S_{\vec{\pi}}$ and let $A^{(0)}_{i}$ and $A^{(1)}_{i}$ be the attach tuples of $T_i$.  The
  hypergraph $P_{\vec{\pi},2\ell}$ is the hypergraph consisting of $T_1,\dots,T_{\ell}$ where the
  vertices of $A^{(1)}_{i}$ are identified with $A^{(0)}_{i+1}$ for every $1 \leq i \leq \ell - 1$.
  (Recall that by definition, $A^{(1)}_i$ and $A^{(0)}_{i+1}$ are tuples (i.e.\ ordered lists) of
  vertices, so the identification of $A^{(1)}_{i}$ and $A^{(0)}_{i+1}$ identifies the corresponding
  vertices in these tuples.) The \emph{attach tuples of $P_{\vec{\pi},2\ell}$} are the tuples
  $A^{(0)}_{1}$ and $A^{(1)}_{\ell}$.
\end{definition}

\begin{figure}[ht] 
\begin{center}
\begin{tikzpicture}[pathscale]
  \node (x1) at (0,0) [vertdia] {};
  \node (x2) at (0,1) [vertdia] {};
  \node (y1) at (1.3,0.5) [vertex] {};
  \node (z1) at (2.6,0) [vertrect] {};
  \node (z2) at (2.6,1) [vertrect] {};
  \node (y2) at (3.9,0.5) [vertex] {};
  \node (x3) at (5.2,0) [vertdia] {};
  \node (x4) at (5.2,1) [vertdia] {};

  \draw[rounded corners=12pt] (x1) -- (1.3,0.5) -- (z1);
  \draw[rounded corners=12pt] (x2) -- (1.3,0.5) -- (z2);
  \draw[rounded corners=12pt] (x3) -- (3.9,0.5) -- (z1);
  \draw[rounded corners=12pt] (x4) -- (3.9,0.5) -- (z2);

  \begin{scope}[xshift=5.2cm]
    \node (x1) at (0,0) [vertdia] {};
    \node (x2) at (0,1) [vertdia] {};
    \node (y1) at (1.3,0.5) [vertex] {};
    \node (z1) at (2.6,0) [vertrect] {};
    \node (z2) at (2.6,1) [vertrect] {};
    \node (y2) at (3.9,0.5) [vertex] {};
    \node (x3) at (5.2,0) [vertdia] {};
    \node (x4) at (5.2,1) [vertdia] {};

    \draw[rounded corners=12pt] (x1) -- (1.3,0.5) -- (z1);
    \draw[rounded corners=12pt] (x2) -- (1.3,0.5) -- (z2);
    \draw[rounded corners=12pt] (x3) -- (3.9,0.5) -- (z1);
    \draw[rounded corners=12pt] (x4) -- (3.9,0.5) -- (z2);
  \end{scope}
\end{tikzpicture}
\caption{$P_{(1,1,1),4}$}
\label{fig:paththree}
\end{center}
\end{figure}
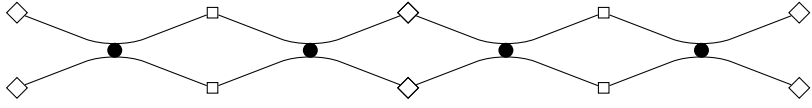 

In Figure~\ref{fig:paththree}, the path $P_{(1,1,1),4}$ is drawn as two copies of $S_{(1,1,1)}$ with
attach tuples identified.  The diamond, circle, and square vertices keep track of the parts
$A,B_2,B_3$.  For a general $P_{(k_1,k_2,k_3),4}$, each diamond vertex is enlarged into $k_1$
vertices, each circle vertex is enlarged into $k_2$ vertices, and each square vertex is enlarged
into $k_3$ vertices.  For a general $\vec{\pi}$, every step can be visualised as in
Figure~\ref{fig:partialsteps} \textit{(b)} as two collections of $k$-partite edges $M_0$ and $M_1$
between $A$ and $B_t$, so all paths $P_{\vec{\pi},2\ell}$ can be visualised as in
Figure~\ref{fig:paththree} as a concatenation of steps.

\begin{onlyarxiv}
\begin{definition}
  Let $\ell \geq 2$.  The \emph{cycle of type $\vec{\pi}$ and length $2\ell$}, denoted
  $C_{\vec{\pi},2\ell}$, is the hypergraph formed from $P_{\vec{\pi},2\ell}$ by identifying the
  attach tuples of $P_{\vec{\pi},2\ell}$.
\end{definition}

\begin{lemma} \label{lem:cycleindpendent}
  If $\pi$ is an (unordered) proper partition of $k$ and $\vec{\pi}$ and $\vec{\pi}'$ are two
  orderings of $\pi$, then $C_{\vec{\pi},2\ell} \cong C_{\vec{\pi}',2\ell}$.
\end{lemma}

\begin{proof} 
Let $\vec{\pi} = (k_1,\dots,k_t)$ and let $\eta : \{2,\dots,t\} \rightarrow \{2,\dots,t\}$ be any
bijection of the numbers $2,\dots,t$.  We first claim that $S_{\vec{\pi}}\cong
S_{(k_1,k_{\eta(2)},\dots,k_{\eta(t)})}$.  This follows directly from the definition of the step;
the bit strings can be permuted using $\eta$.  That is, the isomorphism between $S_{\vec{\pi}}$ and
$S_{(k_1,k_{\eta(2)},\dots,k_{\eta(t)})}$ is the isomorphism which takes a vertex in $A$ in
$V(S_{\vec{\pi}})$ with binary code $a_2\dots a_{t}$ to the vertex with code $a_{\eta^{-1}(2)}\dots
a_{\eta^{-1}(t)}$ in $A$ in $S_{(k_1,k_{\eta(2)},\dots,k_{\eta(t)})}$ and which also takes a vertex
in $B_{j}$ in $S_{\vec{\pi}}$ with code $b_2\dots b_{j-1} b_{j+1}\dots b_t$ to the vertex in
$B_{\eta(j)}$ in $S_{(k_1,k_{\eta(2)},\dots,k_{\eta(t)})}$ with code
$b_{\eta^{-1}(2)}b_{\eta^{-1}(3)}\dots b_{\eta^{-1}(\eta(j)-1)} b_{\eta^{-1}(\eta(j)+1)}
\linebreak[1]\dots b_{\eta^{-1}(t)}$ if $2 < \eta(j) < t$, to the vertex with code
$b_{\eta^{-1}(3)} \dots b_{\eta^{-1}(t)}$ if $\eta(j) = 2$, and to the vertex with code
$b_{\eta^{-1}(2)} \dots b_{\eta^{-1}(t-1)}$ if $\eta(j) = t$.
To see that this bijection preserves hyperedges, let $\{\vec{a},\vec{b}_2,\dots,\vec{b_t}\}$ be a
set of vertices in $S_{\vec{\pi}}$ where $\vec{a} \in A$ and $\vec{b}_j \in B_j$.  For every $j$,
we have $a_2 \dots a_{j-1} a_{j+1}\dots a_t = b_2 \dots b_{j-1} b_{j+1} \dots b_t$ if and only if
$a_{\eta^{-1}(2)} a_{\eta^{-1}(3)} \dots a_{\eta^{-1}(\eta(j)-1)} a_{\eta^{-1}(\eta(j)+1)} \dots
a_{\eta^{-1}(t)} = b_{\eta^{-1}(2)} b_{\eta^{-1}(3)}\dots b_{\eta^{-1}(\eta(j)-1)} \linebreak[1]
b_{\eta^{-1}(\eta(j)+1)} \dots \linebreak[1] b_{\eta^{-1}(t)}$ if $2 < \eta(j) < t$ (similar
statements hold for $\eta(j) = 2$ and $\eta(j) = t$).  This implies that
$\{\vec{a},\vec{b}_2,\dots,\vec{b}_t\}$ is an edge of $S_{\vec{\pi}}$ if and only if the image is an
edge of $S_{(k_1,k_{\eta(2)},\dots,k_{\eta(t)})}$.

By the previous paragraph, it suffices to prove that $C_{\vec{\pi},2\ell} \cong
C_{(k_t,k_2,\dots,k_{t-1},k_1),2\ell}$ to complete the proof of the lemma.  Indeed, if $\vec{\pi}' =
(k_{f(1)}, \dots, k_{f(t)})$ with $f(1) > 1$, then the transformations $(1,\dots,t) \rightarrow
(1,2,\dots,f(1)-1,f(1)+1,\dots,t,f(1)) \rightarrow (f(1),2,\dots,f(1)-1,f(1)+1,\dots,t,1)
\rightarrow (f(1),\dots,f(t))$ show that $C_{\vec{\pi},2\ell} \cong C_{\vec{\pi}',2\ell}$.

The fact that $C_{\vec{\pi},2\ell} \cong C_{(k_t,k_2,\dots,k_{t-1},k_1),2\ell}$ is easy to see in
Figure~\ref{fig:paththree}.  In Figure~\ref{fig:paththree}, consider swapping the diamond and square
vertices.  This changes the path, but the cycle is Figure~\ref{fig:paththree} with the diamond
vertices on the ends identified, so swapping the diamond and square vertices preserves the cycle.
In general, as discussed before, the step $S_{\vec{\pi}}$ can be drawn similar to
Figure~\ref{fig:partialsteps} \textit{(b)} as a collection $M_0$ of $k$-partite edges between
$A^{(0)}$ and $B_t$ and a collection $M_1$ of $k$-partite edges between $A^{(1)}$ and $B_t$ and so
the path $P_{\vec{\pi},2\ell}$ can be visualized like Figure~\ref{fig:paththree}.  Therefore the
cycle $C_{\vec{\pi},2\ell}$ consists of a list of collections of $k$-partite edges;
$E(C_{\vec{\pi},2\ell})$ is $M_{1,0} \dot\cup M_{1,1} \dot\cup M_{2,0} \dot\cup M_{2,1} \dot\cup
\cdots \dot\cup M_{\ell,0} \dot\cup M_{\ell,1}$ where $M_{i,0} \cup M_{i,1}$ is a copy of
$S_{\vec{\pi}}$.  But $M_{i,1} \cup M_{i+1,0}$ (modulo $\ell$) forms a copy of
$S_{(k_t,k_2,\dots,k_{t-1},k_1)}$ for all $i$ as follows.  Since bit strings in the $A$-part of
$M_{i,1}$ have last bit one by definition, drop the last bit.  After dropping these bits, for each
edge $E$ in $M_{i,1}$, the vertices in $E \cap A$ and $E \cap B_t$ have the same code.  Also, for $2
\leq j \leq t-1$ the code for the vertices $E \cap B_j$ is formed by adding a one to the bit string
for $E \cap A$ and then deleting the $(j-1)$th entry.  But since $E \cap A$ and $E \cap B_t$ have
the same code, this is the same as adding a one to the bit string for $E \cap B_t$ and then deleting
the $(j-1)$th entry.  Therefore we could add a one to the end of all the codes in $B_t$ and now have
half of the edges which make up $S_{(k_t,k_2,\dots,k_{t-1},k_1)}$.  A similar argument shows that
$M_{i+1,0}$ forms the other half and so $M_{i,1} \cup M_{i+1,0}$ is a copy of
$S_{(k_t,k_2,\dots,k_{t-1},k_1)}$.  Thus $C_{\vec{\pi},2\ell}$ is built out of $\ell$ copies of
$S_{(k_t,k_2,\dots,k_{t-1},k_1)}$.
\end{proof} 

\begin{definition}
  Let $\pi$ be an (unordered) proper partition of $k$.  The \emph{cycle of type $\pi$ and length
  $2\ell$}, denoted $C_{\pi,2\ell}$, is $C_{\vec{\pi},2\ell}$ where $\vec{\pi}$ is any ordering of
  $\pi$.
\end{definition}
\end{onlyarxiv}

\begin{notarxiv}
  \begin{definition}
    Let $\ell \geq 2$.  The \emph{cycle of type $\pi$ and length $2\ell$}, denoted $C_{\pi,2\ell}$,
    is the hypergraph formed by picking any ordering $\vec{\pi}$ of $\pi$ and identifying the attach
    tuples of $P_{\vec{\pi},2\ell}$.
  \end{definition}

  The definition of $C_{\pi,2\ell}$ is independent of the ordering $\vec{\pi}$; a proof appears
  in~\cite{hqsi-lenz-quasi12-arxiv}.
\end{notarxiv}

\begin{definition}
  Let $\ell \geq 2$.  A \emph{walk of type $\vec{\pi}$ and length $2\ell$ in a hypergraph $H$} is a
  function $f : V(P_{\vec{\pi},2\ell}) \rightarrow V(H)$ that preserves edges.  Informally, a walk
  is a path where the vertices are not necessarily distinct.  A \emph{circuit of type $\pi$ of
  length $2\ell$ in a hypergraph $H$} is a function $f : V(C_{\pi,2\ell}) \rightarrow V(H)$ that
  preserves edges.  Informally, a circuit is a cycle where the vertices are not necessarily
  distinct.
\end{definition}

There are two alternative definitions of the cycle of length four.  First, Conlon et
al.~\cite{hqsi-conlon12} defined a cycle of length four for $\pi = 1 + \dots + 1$ by an operation
called \emph{reflection}.  Our definition of $C_{1 + \dots + 1,4}$ is equivalent to the definition in
\cite{hqsi-conlon12}; this can be seen by noticing that the bit strings in our definition keep track
of the vertex duplications which occur during reflection.

Finally, there is a concise direct definition of the cycle of type $\pi$ and length four which
avoids the complexity of defining steps and paths.  We will not use this shorter definition in this
paper, instead working with steps, paths, and walks, but we include this short definition for
completeness.  Let $D_1,\dots,D_t$ be disjoint sets of size $2^{t-1}$ whose elements are labeled by
$(t-1)$-length binary strings.  The vertex set is $D_1 \dot\cup \dots \dot\cup D_t$.  For $d_1 \in
D_1, \dots, d_t \in D_t$, make $\{d_1,\dots,d_t\}$ a hyperedge if there exists a binary string $s$
of length $t$ such that the code for $d_i$ equals the code formed by deleting the $i$th bit of $s$.
The cycle for general $\pi$ is formed by enlarging this cycle appropriately.
Figure~\ref{fig:altcycledef} shows cycles drawn using this definition.

\begin{figure}[ht] 
\begin{center}
\subfloat[$C_{(1,1),4}$]{%
\begin{tikzpicture}
  \node (v1) at (0,0) [vertex,label=below:0] {};
  \node (v2) at (0,2) [vertex,label=above:0] {};
  \node (v3) at (2,0) [vertex,label=below:1] {};
  \node (v4) at (2,2) [vertex,label=above:1] {};

  \draw (v1) -- node[above,sloped] {00} (v2);
  \draw (v1) -- node[near start,above,sloped] {01} (v4);
  \draw (v2) -- node[near end,above,sloped] {10} (v3);
  \draw (v3) -- node[below,sloped] {11} (v4);
\end{tikzpicture}
}\quad \quad \quad
\subfloat[$C_{(1,1,1),4}$]{%
\begin{tikzpicture}[scale=0.7]
  \tikzstyle{foot}=[font=\footnotesize]

  \node (v00) at (0,0) [vertex,label={[foot]left:00}] {};
  \node (v10) at (3,3) [vertex,label={[foot]above:10}] {};
  \node (v11) at (6,0) [vertex,label={[foot]right:11}] {};
  \node (v01) at (3,-3) [vertex,label={[foot]below:01}] {};
 
  \draw (0,-3) +(45:0.5) node (s00) [vertrect,label={[foot]north east:00}] {};
  \draw (0,-3) +(225:0.5) node (s10) [vertrect,label={[foot]south west:10}] {};
  \draw (6,3) +(45:0.5) node (s11) [vertrect,label={[foot]north east:11}] {};
  \draw (6,3) +(225:0.5) node (s01) [vertrect,label={[foot]south west:01}] {};

  \draw (0,3) +(135:0.5) node (d10) [vertdia,label={[foot]north west:10}] {};
  \draw (0,3) +(-45:0.5) node (d00) [vertdia,label={[foot]south east:00}] {};
  \draw (6,-3) +(135:0.5) node (d01) [vertdia,label={[foot]north west:01}] {};
  \draw (6,-3) +(-45:0.5) node (d11) [vertdia,label={[foot]south east:11}] {};

  \draw (d10).. controls (v00).. node[left,very near end,foot] {100} (s10);
  \draw (s00).. controls (v00).. node[right,near end,foot] {000} (d00);
  \draw (s10).. controls (v01).. node[below,very near start,foot] {101} (d11);
  \draw (s00).. controls (v01).. node[above,near end,foot] {001} (d01);
  \draw (d01).. controls (v11).. node[left,near start,foot] {011} (s01);
  \draw (d11).. controls (v11).. node[right,very near end,foot] {111} (s11);
  \draw (d00).. controls (v10).. node[below,near end,foot] {010}  (s01);
  \draw (d10).. controls (v10).. node[above,very near start,foot] {110} (s11);
\end{tikzpicture}
}
\caption{Alternate definition of the cycle of length four}
\label{fig:altcycledef}
\end{center}
\end{figure}
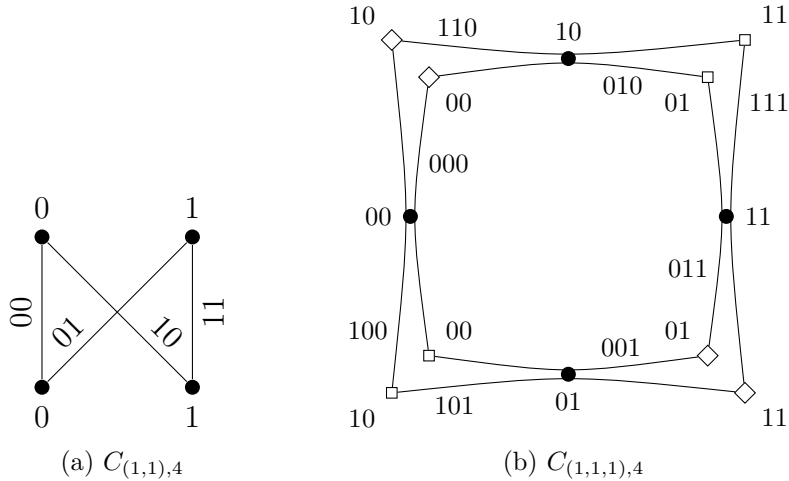 

\section{Hypergraph Eigenvalues}
  \label{sec:hypergrapheigenvalues}

This section contains the definition of the largest and second largest eigenvalues of a hypergraph with
respect to $\pi$ and also contains some discussion and basic facts about them.

There have been three independently developed approaches to hypergraph eigenvalues: a definition by
Chung~\cite{ee-chung93} and Lu and Peng~\cite{ee-lu11,ee-lu11-2} using matrices, an approach of
Friedman and Wigderson~\cite{ee-friedman95,ee-friedman95-2} and Cooper and Dutle~\cite{ee-cooper11},
and lastly the eigenvalues of the shadow graph~\cite{ee-bilu04, ee-feng96, ee-martinez00,
ee-martinez01, ee-storm06}.  The definitions of Friedman and
Wigderson~\cite{ee-friedman95,ee-friedman95-2} are most suitable for our purposes and we will use
their definitions as our starting point.

\begin{definition}
  Let $V_1,\dots,V_k$ be finite-dimensional vector spaces over $\mathbb{R}$.  A \emph{$k$-linear
  map} is a function $\phi : V_1 \times \dots \times V_k \rightarrow \mathbb{R}$ such that for
  each $1 \leq i \leq k$, $\phi$ is linear in the $i$th coordinate.  That is, for every fixed $x_i
  \in V_i$, $\phi(x_1,\dots,x_{i-1},\cdot,x_{i+1},\dots,x_{n})$ is a linear map from $V_i$ to
  $\mathbb{R}$.  A $k$-linear map $\phi : V^k \rightarrow \mathbb{R}$ is \emph{symmetric} if for
  all permutations $\eta$ of $[k]$ and all $x_1,\dots,x_k \in V$, $\phi(x_1,\dots,x_k) =
  \phi(x_{\eta(1)}, \dots, x_{\eta(k)})$.
\end{definition}

\begin{definition}
  Let $V_1, \dots, V_k$ be finite-dimensional vector spaces over $\mathbb{R}$, let $B_i =
  \{b_{i,1},\dots,\linebreak[1] b_{i,\dim(V_i)}\}$ be an orthonormal basis of $V_i$, and let $\phi :
  B_1 \times \dots \times B_k \rightarrow \mathbb{R}$ be any map.  \emph{Extending $\phi$ linearly
  to $V_1 \times \dots \times V_k$} means that $\phi$ is extended to a map $V_1 \times \dots \times
  V_k \rightarrow \mathbb{R}$ where for $x_1 \in V_1, \dots, x_k \in V_k$,
  \begin{align} \label{eq:extendlinearly}
    \phi(x_1,\dots,x_k) = \sum_{j_1 = 1}^{\dim(V_1)} \cdots \sum_{j_k = 1}^{\dim(V_k)}
    \dotp{x_1}{b_{1,j_1}} \cdots \dotp{x_k}{b_{k,j_k}} \phi(b_{1,j_1},\dots,b_{k,j_k}).
  \end{align}
  Note that extending $\phi$ in this way produces a $k$-linear map.
\end{definition}

\begin{definition} \textbf{(Friedman and
  Wigderson~\cite{ee-friedman95,ee-friedman95-2})} Let $H$ be a $k$-uniform hypergraph with loops.
  The \emph{adjacency map of $H$} is the symmetric $k$-linear map $\tau_H : W^k \rightarrow
  \mathbb{R}$ defined as follows, where $W$ is the vector space over $\mathbb{R}$ of dimension
  $|V(H)|$.  First, for all $v_1, \dots, v_k \in V(H)$, let
  \begin{align*}
    \tau_H(e_{v_1}, \dots, e_{v_k}) = 
    \begin{cases}
      1 & \left\{ v_1, \ldots, v_k \right\} \in E(H), \\
      0 & \text{otherwise},
    \end{cases}
  \end{align*}
  where $e_v$ denotes the indicator vector of the vertex $v$, that is the vector which has a one in
  coordinate $v$ and zero in all other coordinates.  We have defined the value of $\tau_H$ when the
  inputs are standard basis vectors of $W$.  Extend $\tau_H$ to all the domain linearly.
\end{definition}

\begin{definition}
  Let $W_1,\dots,W_k$ be finite dimensional vector spaces over $\mathbb{R}$, let $\left\lVert \cdot
  \right\rVert$ denote the Euclidean $2$-norm on $W_i$, and
  let $\phi : W_1 \times \dots \times W_k \rightarrow \mathbb{R}$ be a $k$-linear map.  The
  \emph{spectral norm of $\phi$} is
  \begin{align*}
    \left\lVert \phi \right\rVert 
      = \sup_{\substack{x_i \in W_i \\ \left\lVert x_i \right\rVert = 1}} \left|
          \phi(x_1,\dots,x_k) \right|.
  \end{align*}
\end{definition}

Before defining the first and second largest eigenvalue of $H$ with respect to a general partition
$\pi$, we give the definitions when $\pi = 1 + \dots + 1$, that is $\pi$ is the partition into $k$
ones.

\begin{definition}
  Let $H$ be an $n$-vertex, $k$-uniform hypergraph, let $W$ be the vector space over $\mathbb{R}$ of
  dimension $n$, and let $J : W^k \rightarrow \mathbb{R}$ be the all-ones map.  That is, if
  $e_{i_1},\dots,e_{i_k}$ are any standard basis vectors of $W$, then $J(e_{i_1},\dots,e_{i_k}) =
  1$, and $J$ is extended linearly to all of the domain as in \eqref{eq:extendlinearly}.
  
  The \emph{largest eigenvalue of $H$ with respect to $\pi = 1 + \dots + 1$} is $\left\lVert \tau_H
  \right\rVert$ and the \emph{second largest eigenvalue of $H$ with respect to $\pi = 1 + \dots + 1$}
  is $\left\lVert \tau_H - \frac{k!|E(H)|}{n^k} J \right\rVert$.
\end{definition}

In order to extend this definition to general $\vec{\pi} = (k_1,\dots,k_t)$, it is convenient to use
the language of tensor products.

\begin{definition}
  Let $V$ and $W$ be finite dimensional vector spaces over $\mathbb{R}$ of dimension $n$ and $m$
  respectively.  The \emph{tensor product of $V$ and $W$}, written $V \otimes W$, is the vector
  space over $\mathbb{R}$ of dimension $nm$.  A typical tensor $a$ in $V \otimes W$ has the form $a
  = \sum_{i=1}^{\dim(V)} \sum_{j=1}^{\dim(W)} \alpha_{i,j} (e_i \otimes e'_j)$, where $\alpha_{i,j}
  \in \mathbb{R}$ and $e_1,\dots e_{\dim(V)}$ is the standard basis of $V$ and
  $e'_1,\dots,e'_{\dim(W)}$ is the standard basis of $W$.  The \emph{length} of a tensor is the
  length of the vector in the vector space $V \otimes W$.  Thus the length of $a$ is $\left(
  \sum_{i=1}^{\dim(V)} \sum_{j=1}^{\dim(W)} \alpha_{i,j}^2 \right)^{1/2}$.
\end{definition}

We are now ready to define the map $\tau_{\vec{\pi}}$ and then the first and second largest
eigenvalue of $H$ with respect to $\pi$ for a general $\pi$.  In the definition, think of the tensor
product $W^{\otimes k_i}$ as a vector space of dimension $|V(H)|^{k_i}$ indexed by ordered
$k_i$-sets of vertices.

\begin{definition}
  Let $W$ be a finite dimensional vector space over $\mathbb{R}$, let $\sigma : W^k \rightarrow
  \mathbb{R}$ be any $k$-linear function, and let $\vec{\pi}$ be a proper ordered partition of $k$,
  so $\vec{\pi} = (k_1,\dots,k_t)$ for some integers $k_1,\dots,k_t$ with $t \geq 2$.  Now define a
  $t$-linear function $\sigma_{\vec{\pi}} : W^{\otimes k_1} \times \dots \times W^{\otimes k_t}
  \rightarrow \mathbb{R}$ by first defining $\sigma_{\vec{\pi}}$ when the inputs are basis vectors
  of $W^{\otimes k_i}$ and then extending linearly.  For each $i$, $B_i = \{ b_{i,1} \otimes \cdots
  \otimes b_{i,k_i} : b_{i,j} \, \text{is a standard basis vector of } $W$ \}$ is a basis of
  $W^{\otimes k_i}$, so for each $i$, pick $b_{i,1} \otimes \cdots \otimes b_{i,k_i} \in B_i$ and
  define
  \begin{align*}
    \sigma_{\vec{\pi}} \left( b_{1,1} \otimes \dots \otimes b_{1,k_1}, \dots,
                              b_{t,1} \otimes \dots \otimes b_{t,k_t} \right) 
     \stocnl = \sigma(b_{1,1},\dots,b_{1,k_1},\dots,b_{t,1},\dots,b_{t,k_t}).
  \end{align*}
  Now extend $\sigma_{\vec{\pi}}$ linearly to all of the domain.
  $\sigma_{\vec{\pi}}$ will be $t$-linear since $\sigma$ is $k$-linear.
\end{definition}

\begin{definition}
  Let $H$ be a $k$-uniform hypergraph with loops and let $\tau = \tau_H$ be the ($k$-linear)
  adjacency map of $H$.  Let $\pi$ be any (unordered) partition of $k$ and let $\vec{\pi}$ be any
  ordering of $\pi$.  The \emph{largest and second largest eigenvalues of $H$ with respect to
  $\pi$}, denoted $\lambda_{1,\pi}(H)$ and $\lambda_{2,\pi}(H)$, are defined as
  \begin{align*}
    \lambda_{1,\pi}(H) := \left\lVert \tau_{\vec{\pi}} \right\rVert
    \quad \text{and} \quad
    \lambda_{2,\pi}(H) := \left\lVert \tau_{\vec{\pi}} - \frac{k!|E(H)|}{n^k} J_{\vec{\pi}}
    \right\rVert.
  \end{align*}
\end{definition}

Both $\lambda_{1,\pi}(H)$ and $\lambda_{2,\pi}(H)$ are well defined since for any two orderings
$\vec{\pi}$ and $\vec{\pi}'$ of $\pi$, $\tau_{\vec{\pi}} = \tau_{\vec{\pi}'}$ and $J_{\vec{\pi}} =
J_{\vec{\pi}'}$ since both $\tau$ and $J$ are symmetric maps.

\begin{remarks} \hspace{1cm}

\begin{itemize}


  \item For a graph $G$ ($k = 2$ and $\pi = 1 + 1$), $\lambda_{1,1+1}(G)$ equals the largest
    eigenvalue in absolute value of the adjacency matrix $A$ of $G$ since both are equal to $\sup
    \left\{ |x^T A x| : \left\lVert x \right\rVert = 1\right\}$. Additionally, if $G$ is $d$-regular,
    then $\lambda_{2,1+1}(G)$ equals the second largest eigenvalue of $A$ in absolute value.
    Indeed, if $G$ is a $d$-regular graph, then $2|E(H)|/n^2 = \frac{d}{n}$, so $\lambda_{2,1+1}(G)
    = \left\lVert \tau_G - \frac{d}{n}J \right\rVert$. The bilinear map $\tau_G - \frac{d}{n} J$
    corresponds to the matrix $A - \frac{d}{n}J$ where $J$ is now the all-ones matrix. The largest
    eigenvalue of $A - \frac{d}{n}J$ in absolute value is the second largest eigenvalue of $A$ in
    absolute value, and this equals the spectral norm of the respective map.

  \item For any $k$-uniform hypergraph $H$, $\lambda_{1,1+\dots+1}(H)$ exactly matches the
    definition of Friedman and Wigderson~\cite{ee-friedman95,ee-friedman95-2}. \cite{ee-friedman95,
    ee-friedman95-2} did not define the second largest eigenvalue for all hypergraphs.  For
    $d$-coregular hypergraphs with loops, \cite{ee-friedman95,ee-friedman95-2} defined the second
    largest eigenvalue and it exactly corresponds to our definition of $\lambda_{2,1+\dots+1}(H)$,
    where $\frac{k!|E(H)|}{n^k} = \frac{d}{n}$ (recall that $H$ has loops which is why $n^k$ appears
    in the denominator instead of the falling factorial).  For the random hypergraph $G^{(k)}(n,p)$,
    \cite{ee-friedman95,ee-friedman95-2} also defined a second largest eigenvalue with respect to
    density $p$ as the spectral norm of $\tau_{G(n,p)} - p J$.  While different than our definition,
    $p = (1+o(1))\frac{k!|E(G(n,p))|}{n^k}$ so the definitions are similar.

  \item If $H$ is a $k$-uniform, $d$-coregular hypergraph with loops, Friedman and
    Wigderson~\cite{ee-friedman95, ee-friedman95-2} proved several facts about
    $\lambda_{1,1+\dots+1}(H)$ and $\lambda_{2,1+\dots+1}(H)$.  First, $\lambda_{1,1 + \dots + 1}(H)
    = dn^{(k-2)/2}$ and the supremum is achieved by the all-ones vectors scaled to unit length.
    They also proved several facts about $\lambda_{2,1 + \dots + 1}(H)$ including upper and lower
    bounds, an Expander Mixing Lemma which we generalize to all $\pi$ in
    Theorem~\ref{thm:expandertensormixing}, and the asymptotic value of $\lambda_{2,1 + \dots +
    1}(G(n,p))$.
\end{itemize}
\end{remarks}

\section{\propeig{$\pi$} $\Rightarrow$ \propexpand{$\pi$}}
  \label{sec:eigen2expander}

In this section we prove a generalization of the graph Expander Mixing Lemma which relates spectral
and expansion properties of graphs.  The graph version was first discovered independently by Alon
and Milman~\cite{ee-alon95} and Tanner~\cite{ee-tanner84}.  For background on graph expansion and
eigenvalues, see \cite{ee-alon86,prob-alonspencer,ee-survey06}.  The following theorem extends the
hypergraph Expander Mixing Lemma of Friedman and Widgerson~\cite{ee-friedman95,ee-friedman95-2},
which applied for $\pi = 1 + \dots + 1$.  The theorem is stated for ordered partitions $\vec{\pi}$,
but trivially gives the same result for any ordering $\vec{\pi}$ of a partition $\pi$.

\begin{thm} (Hypergraph Expander Mixing Lemma) \label{thm:expandertensormixing}
  Let $H$ be an $n$-vertex, $k$-uniform hypergraph with loops.  Let $\vec{\pi} = (k_1,\dots,
  \linebreak[1] k_t)$ be a proper ordered partition of $k$ and let $S_i \subseteq \binom{V(H)}{k_i}$
  for $1 \leq i \leq t$ (where the elements of $S_i$ are potentially multisets of size $k_i$).  Then
  \begin{align*}
    \left| e(S_1,\dots,S_t) - \frac{k!|E(H)|}{n^k} \prod_{i=1}^t |S_i| \right| \leq \lambda_{2,\pi}(H)
  \sqrt{|S_1| \cdots |S_t|},
  \end{align*}
  where $e(S_1,\dots,S_t)$ is the number of ordered tuples $(s_1, \dots, \linebreak[1] s_t)$ such that $s_1 \cup
  \dots \cup s_t \in E(H)$ and $s_i \in S_i$.
\end{thm}

\begin{proof} 
Let $q = \frac{k!|E(H)|}{n^k}$, let $\tau_H$ be the adjacency map of $H$, and let $\sigma = \tau_H -
q J$.  It is easy to see that by definition, $(\tau - q J)_{\vec{\pi}} = \tau_{\vec{\pi}} - q
J_{\vec{\pi}}$, so $\lambda_{2,\pi}(H) = \left\lVert \sigma_{\vec{\pi}} \right\rVert$.  Let
$\chi_{S_i} \in W^{k_i\otimes}$ be the indicator tensor of $S_i$.  If we let $V(H) = [n]$, then
\begin{align*}
  \chi_{S_i} = \sum_{\substack{\{v_1, \dots, v_{k_i} \} \in S_i \\
                              v_1 \leq \dots \leq v_{k_i} }}
                  (e_{v_1} \otimes \dots \otimes e_{v_{k_i}}).
\end{align*}
By the linearity of $\sigma_{\vec{\pi}}$ and the definition of $J_{\vec{\pi}}$,
\begin{align*}
  \sigma_{\vec{\pi}}(\chi_{S_1},\dots,\chi_{S_t}) &=
  \tau_{\vec{\pi}}(\chi_{S_1},\dots,\chi_{S_t}) - q J_{\vec{\pi}}(\chi_{S_1},\dots,\chi_{S_t})
  \stocnla = e(S_1,\dots,S_t) - q \prod_{i=1}^t |S_i|.
\end{align*}
Before upper bounding this by $\lambda_{2,\pi}(H)$, we must scale each indicator tensor to be unit
length.  Since $\{e_{j_1} \otimes \dots \otimes e_{j_{k_i}} : 1 \leq j_1,\dots,j_{k_i} \leq n \}$
forms a basis of $W^{\otimes k_i}$, we have $\left\lVert \chi_{S_i} \right\rVert = \sqrt{|S_i|}$.
Thus
\begin{align*}
  \left| \sigma_{\vec{\pi}}\left(\frac{\chi_{S_1}}{\left\lVert \chi_{S_1} \right\rVert}, \dots,
                         \frac{\chi_{S_t}}{\left\lVert \chi_{S_t} \right\rVert}\right) \right| \leq
                         \left\lVert \sigma_{\vec{\pi}} \right\rVert
                         = \lambda_{2,\pi}(H).
\end{align*}
Consequently,
\begin{align*}
  \left|\sigma_{\vec{\pi}}(\chi_{S_1},\dots,\chi_{S_t})\right|
  &\leq \lambda_{2,\pi}(H) \left\lVert \chi_{S_1} \right\rVert \cdots \left\lVert \chi_{S_t} \right\rVert
  \stocnla = \lambda_{2,\pi}(H) \sqrt{\left| S_1 \right|\cdots \left| S_t \right|},
\end{align*}
and the proof is complete.
\end{proof} 

\begin{lemma} \label{lem:lambda1boundimpliesdense}
  Let $\mathcal{H} = \{H_n\}$ be a sequence of $k$-uniform hypergraphs with loops with $|V(H_n)| =
  n$ and $|E(H_n)| \geq p \binom{n}{k} + o(n^k)$.  Let $\tau_n$ be the adjacency map of $H_n$ and
  let $\vec{\pi} = (k_1,\dots,k_t)$ be a proper ordered partition of $k$.  If $\lambda_{1,\pi}(H_n)
  = pn^{k/2} + o(n^{k/2})$, then $|E(H_n)| = p \binom{n}{k} + o(n^k)$.
\end{lemma}

\begin{proof} 
Throughout this proof the subscripts on $n$ are dropped for simplicity.  Let $W$ be the vector space
over $\mathbb{R}$ of dimension $n$.  For $1 \leq i \leq t$, let $\vec{1}_{k_i}$ denote the all-ones
vector in $W^{\otimes k_i}$, so $\lVert \vec{1}_{k_i} \rVert = n^{k_i/2}$.  Then
\begin{align*}
  \tau_{\vec{\pi}}\left( \frac{\vec{1}_{k_1}}{n^{k_1/2}}, \dots, \frac{\vec{1}_{k_t}}{n^{k_t/2}} \right)
  &= \frac{1}{n^{k/2}} \tau_{\vec{\pi}}\left( \vec{1}_{k_1}, \dots, \vec{1}_{k_t} \right)
  \stocnla = \frac{1}{n^{k/2}} \tau(\vec{1}_1, \dots, \vec{1}_1) \\
  &= \frac{1}{n^{k/2}} \sum_{i_1,\dots,i_k=1}^n \tau(e_{i_1},\dots,e_{i_k}) \\
  &= \frac{1}{n^{k/2}} k! |E(H)|.
\end{align*}
Thus the spectral norm of $\tau_{\vec{\pi}}$ is at least $k!|E(H)|/n^{k/2}$, so
\begin{align*}
  pn^{k/2} \leq \frac{k!|E(H)|}{n^{k/2}} + o(n^{k/2})
  &\leq \left\lVert \tau_{\vec{\pi}} \right\rVert + o(n^{k/2}) 
  \stocnla = pn^{k/2} + o(n^{k/2})
\end{align*}
This implies equality (up to $o(n^{k/2})$) throughout the above expression.  In particular,
$|E(H_n)| \linebreak[1] = p \binom{n}{k} + o(n^{k})$.
\end{proof} 

\begin{proof}[\prooftext that \propeig{$\pi$} $\Rightarrow$ \propexpand{$\pi$}] 
First, \propeig{$\pi$} contains the assertion that $\lambda_{1,\pi}(H_n) = pn^{k/2} + o(n^{k/2})$
which by Lemma~\ref{lem:lambda1boundimpliesdense} implies $|E(H_n)| = p\binom{n}{k} + o(n^{k})$.
Consequently, $k! |E(H_n)|/n^{k} = (1+o(1)) p$ and Theorem~\ref{thm:expandertensormixing} imply that
\begin{align} \label{eq:eigtoexpand}
  \Big| e(S_1,\dots,S_t) - (1+o(1))p \prod_{i=1}^t |S_i| \Big| \leq \lambda_{2,\pi}(H)
  \sqrt{|S_1| \cdots |S_t|}
\end{align}
for any choice of $S_i \subseteq \binom{V(H_n)}{k_i}$, $i = 1, \dots, t$.
Since $\pi$ is a partition of $k$, $\sqrt{|S_1|\cdots|S_t|} = O(n^{k/2})$.  Also, \propeig{$\pi$}
states that $\lambda_{2,\pi}(H) = o(n^{k/2})$.  Thus \eqref{eq:eigtoexpand} becomes
\begin{align*}
  \Big| e(S_1,\dots,S_t) - p |S_1| \cdots |S_t| \Big|  = o(n^k),
\end{align*}
which proves \propexpand{$\pi$}.
\end{proof} 

\section{\propexpand{$\pi$} $\Rightarrow$ \propcount{$\pi$-\texttt{linear}}}
  \label{sec:expander2c4}

The proof that \propexpand{$\pi$} $\Rightarrow$ \propcount{$\pi$-linear} follows from an embedding
lemma for hypergraphs.  
\begin{notarxiv}
  The proof of Proposition~\ref{prop:embeddinglemma} below is a generalization of an argument by
  Kohayakawa et al.~\cite{hqsi-kohayakawa10} who proved it in the special case of linear
  hypergraphs, so we omit the proof.  A detailed proof appears
  online~\cite{hqsi-lenz-quasi12-arxiv}.
\end{notarxiv}
\begin{onlyarxiv}
  The proof of Proposition~\ref{prop:embeddinglemma} below is a generalization of an argument by
  Kohayakawa et al.~\cite{hqsi-kohayakawa10} who proved it in the special case of linear
  hypergraphs.
\end{onlyarxiv}
The proposition below
is stated for ordered partitions $\vec{\pi}$, but it is easy to see that the proposition is
independent of the ordering chosen for $\vec{\pi}$.

\begin{prop} \label{prop:embeddinglemma}
  Let $\vec{\pi} = (k_1,\dots,k_t)$ be a proper ordered partition of $k$, let $0 < p < 1$, and let $F$ be
  any fixed $k$-uniform, $\pi$-linear hypergraph with $f$ vertices and $m$ edges.

  Let $\mathcal{H} = \{H_n\}_{n\rightarrow\infty}$ be a sequence of $k$-uniform hypergraphs with
  loops with $|V(H_n)| = n$, $|E(H_n)| = p\binom{n}{k} + o(n^k)$, and for which
  \propexpand{$\vec{\pi}$} holds.  In other words, for every $S_1 \subseteq \binom{V(H)}{k_1},\dots,
  \linebreak[1] S_t \subseteq \binom{V(H)}{k_t}$, we have $e(S_1,\dots,S_t) = p |S_1| \cdots |S_t| +
  o(n^k)$.  Then the number of labeled copies of $F$ in $H$ is $p^m n^f + o(n^f)$.
\end{prop}

\begin{onlyarxiv}
\begin{proof} 
The proof is by induction on the number of edges of $F$.  If $F$ has zero or one edge, then the
result is trivial.  So assume that $F$ has at least two edges and let $E$ be the last edge in the
ordering provided by the $\pi$-linearity of $F$.  Let $F_*$ be the hypergraph formed by deleting all
vertices of $E$ from $F$.  Let $Q_*$ be a labeled copy of $F_*$ in $H$ by which we mean that $Q_*$
is an injective edge preserving map $Q_* : V(F_*) \rightarrow V(H)$.  We can count the number of
labeled copies of $F$ in $H$ by counting for each $Q_*$, the number of ways $Q_*$ extends to a
labeled copy of $F$.  More precisely, we count the number of edge preserving injections $Q : V(F)
\rightarrow V(H)$ which when restricted to $V(F_*)$ match the injection $Q_*$.  Since $F$ is
$\pi$-linear, the edge $E$ can be divided into $A_1,\dots,A_t$ such that $|A_i| = k_i$ and the edges
of $F$ intersecting $E$ can be divided into sets $R_1,\dots,R_t$ such that every edge in $R_i$
intersects $E$ in a subset of $A_i$.

Consider some $Q_*$ in $H$; we will count how many ways it extends to a labeled copy of $F$.  For $1
\leq i \leq t$, define $S_i(Q_*)$ to be the following collection of $k_i$-sets. Let $Y \subseteq
V(H)$ be a set of $k_i$ vertices and add $Y$ to $S_i(Q_*)$ if $Y \cap Im(Q_*) = \emptyset$ and there
exists an edge preserving injection $V(F_*) \cup A_i \rightarrow Im(Q_*) \cup Y$ which when
restricted to $V(F_*)$ matches the map $Q_*$.  More informally, $S_i(Q_*)$ consists of all
$k_i$-sets $Y$ of vertices which can be used to extend $Q_*$ to embed a labeled copy of $F_* \cup
R_i$.

Every edge counted by $e(S_1(Q_*), \dots, S_t(Q_*))$ creates several labeled copies of $F$
which extend $Q_*$.  First, let $\Delta_i$ be the number of edge preserving bijections $V(F_*) \cup
A_i \rightarrow V(F_*) \cup A_i$ which are the identity map when restricted to $V(F_*)$.  More
informally, if we are given a non-labeled $F_* \cup R_i$ together with a labeling of the vertices of
$F_*$, $\Delta_i$ is the number of ways of labeling the vertices of $A_i$.  The numbers
$\Delta_1,\dots,\Delta_t$ are fixed numbers depending only on $F$; $\Delta_i$ depends on the way
that edges in $R_i$ intersect.  For example, if every edge in $R_i$ meets $E$ in exactly $A_i$, then
$\Delta_i = k_i! $.  If the edges in $R_i$ intersect differently, $\Delta_i$ will change but still
depend only on $F$.  Now we count the number of edge preserving injections $Q : V(F) \rightarrow V(H)$
where $Q|_{V(F_*)} = Q_*$ as follows.  First pick an edge to use for $E$; this consists of picking
one of the edges which take a $k_1$-set $Y_1$ from $S_1(Q_*)$, a $k_2$-set $Y_2$ from $S_2(Q_*)$,
and so on.  There are exactly $e(S_1(Q_*),\dots,S_t(Q_*))$ such edges.  We are embedding a labeled
copy of $F$ so next we order the vertices inside the set $Y_1$ chosen from $S_1(Q_*)$; there are
$\Delta_1$ ways of ordering the vertices of $Y_1$.  Similarly, we order the vertices inside the
other sets chosen from $S_i(Q_*)$ for a total of $\prod \Delta_i$ orderings.  These are all the
labeled copies of $F$ which extend $Q_*$, so there are exactly $e(S_1(Q_*),\dots,S_t(Q_*)) \prod
\Delta_i$ labeled copes of $F$ extending $Q_*$, i.e.\ exactly $e(S_1(Q_*),\dots,S_t(Q_*)) \prod
\Delta_i$ edge preserving injections $V(F) \rightarrow V(H)$ which when restricted to $V(F_*)$ are
$Q_*$.  By assupmtion,
\begin{align*}
   e\big(S_1(Q_*), \dots, S_t(Q_*)\big) =  p |S_1(Q_*)| \cdots |S_t(Q_*)|  + o(n^k).
\end{align*}
Therefore, we can count the number of labeled copies of $F$ in $H$ by summing the value of
$e(S_1(Q_*),\linebreak[1]\dots,S_t(Q_*))\prod \Delta_i$ over all $Q_*$ in $H$.  By the notation
$\#\{F \, \text{in} \, H\}$ we mean the number of labeled copies of $F$ in $H$.
\begin{align}
  \#\{F \, \text{in} \, H\} &= \sum_{Q_*} e(S_1(Q_*),\dots,S_t(Q_*)) \prod_i \Delta_i \nonumber \\
    &= \sum_{Q_*}
  \left(p \left| S_1(Q_*) \right| \cdots \left| S_t(Q_*) \right| \prod_i \Delta_i + o(n^k) \right)
  \nonumber \\
  &= p \sum_{Q_*}|S_1(Q_*)| \cdots |S_t(Q_*)| \prod_i \Delta_i + o(n^f), \label{eq:ex2c4prodsizesi}
\end{align}
since the number of $Q_*$ is bounded above by $n^{|V(F_*)|} = n^{f-k}$ so summing the term $o(n^k)$
over $Q_*$ is bounded by $o(n^f)$.  Let $F_-$ be the hypergraph formed by removing the edge $E$ from
$F$ but keeping the same vertex set.  Then similarly to the above, the number of labeled copies of
$F_-$ extending $Q_*$ is $|S_1(Q_*)|\cdots|S_t(Q_*)| \prod \Delta_i$ since every labeled copy of
$F_-$ is formed by picking a set $Y_1$ from $S_1(Q_*)$, ordering it in $\Delta_i$ ways, picking a
set $Y_2$ from $S_2(Q_*)$ and ordering it, and so on.  Therefore the number of labeled copies of
$F_-$ in $H$ is counted by summing over $Q_*$ and counting $|S_1(Q_*)|\cdots|S_t(Q_*)| \prod
\Delta_i$. Thus \eqref{eq:ex2c4prodsizesi} continues as
\begin{align}
  \#\{F \, \text{in} \, H\} &= p \, \#\{F_- \,\,\, \text{in} \, H\} + o(n^f).
  \label{eq:ex2c4countFminus}
\end{align}
By hypothisis, $F$ is $\pi$-linear so $F_-$ is $\pi$-linear.  Thus by induction the number of
labeled copies of $F_-$ in $H$ is $p^{m-1} n^{f} + o(n^f)$.  Inserting this into
\eqref{eq:ex2c4countFminus} shows that the number of labeled copies of $F$ in $H$ is $p^{m} n^{f} + o(n^f)$.
\end{proof} 
\end{onlyarxiv}

\section{\texttt{Cycle}$_{4\ell}$[$\pi$] $\Rightarrow$ \propeig{$\pi$}}
  \label{sec:c4toeigenvalue}

\newcommand{\tr}[1]{\text{Tr}\left[#1\right]}

In this section, we prove that if $\mathcal{H}$ is a sequence of $d$-coregular, $k$-uniform
hypergraphs with loops which satisfies \propcycle[4\ell]{$\pi$}, then $\mathcal{H}$ satisfies
\propeig{$\pi$}.  Indeed, if $H$ is $d$-coregular with loops, then $\lambda_{1,\pi}(H) = dn^{k/2-1}$
and the vectors maximizing $\tau_{\vec{\pi}}$ are the all-ones vectors scaled to unit length
(see~\cite{ee-friedman95,ee-friedman95-2}).  These facts simplify the proof of
\propcycle[4\ell]{$\pi$} $\Rightarrow$ \propeig{$\pi$} which appears in this section.  In a
companion paper~\cite{hqsi-lenz-quasi12-nonregular}, we develop the additional algebra required to
prove \propcycle[4\ell]{$\pi$} $\Rightarrow$ \propeig{$\pi$} for all sequences.  Throughout this
section, let $0 < p < 1$ be a fixed integer and define $d = d(n) = \left\lfloor pn \right\rfloor$.

First, let us recall the proof of \propcycle{$1+1$} $\Rightarrow$ \propeig{$1+1$} for graphs.  Let
$A$ be the adjacency matrix of a $d$-regular graph $G$. Then $\tr{A^4}$ is the number of circuits of
length $4$ so \propcycle{$1+1$} implies that $\tr{A^4} = d^4 + o(n^4)$.  Since $G$ is $d$-regular,
the largest eigenvalue of $A^4$ is $d^4$ so that all eigenvalues of $A$ besides $d$ are $o(n)$ in
absolute value, completing the proof that \propeig{$1+1$} holds.  Our proof for hypergraphs follows
the same outline once some algebraic facts about multilinear maps are proved.  In
Section~\ref{sub:c4toeigmultiproduct}, we define (non-standard) products and powers of multilinear
maps.  In Section~\ref{sub:c4toeigpowerscountwalks}, we show that the powers of multilinear maps
count walks and that the trace of the powers of multilinear maps counts circuits.  Finally,
Section~\ref{sub:c4toeigfinalproof} contains the proof that \propcycle[4\ell]{$\pi$} $\Rightarrow$
\propeig{$\pi$}.



\subsection{Products and powers of multilinear maps}
\label{sub:c4toeigmultiproduct}

In this section, we give (non-standard) definitions of the products and powers of multilinear maps.

\begin{definition}
  Let $V_1,\dots,V_t$ be finite dimensional vector spaces over $\mathbb{R}$ and let $\phi,\psi : V_1
  \times \dots \times V_t \rightarrow \mathbb{R}$ be $t$-linear maps.  The \emph{product} of $\phi$
  and $\psi$, written $\phi \ast \psi$, is a $(t-1)$-linear map defined as follows.  Let
  $u_1,\dots,u_{t-1}$ be vectors where $u_i \in V_i$. Let $\{b_1,\dots,b_{\dim(V_t)}\}$ be any
  orthonormal basis of $V_t$.
  \begin{gather*}
    \phi \ast \psi : (V_1 \otimes V_1) \times (V_2 \otimes V_2) \times \dots
                     \times (V_{t-1} \otimes V_{t-1}) \rightarrow \mathbb{R} \\
    \phi \ast \psi(u_1 \otimes v_1, \dots, u_{t-1} \otimes v_{t-1}) 
       \stocnl := \sum_{j=1}^{\dim(V_t)} \phi(u_1,\dots,u_{t-1},b_j) \psi(v_1,\dots,v_{t-1},b_j)
  \end{gather*}
  Extend the map $\phi \ast \psi$ linearly to all of the domain to produce a $(t-1)$-linear map.
\end{definition}

It is straigtforward to see that the above definition is well defined: the map is the same for any
choice of orthonormal basis by the linearity of $\phi$ and $\psi$.  A proof of this fact appears
in~\cite{hqsi-lenz-quasi12-nonregular}.

\begin{definition}
  Let $V_1,\dots,V_t$ be finite dimensional vector spaces over $\mathbb{R}$ and let $\phi : V_1
  \times \dots \times V_t \rightarrow \mathbb{R}$ be a $t$-linear map and let $s$ be an integer $0
  \leq s \leq t-1$.  Define
  \begin{gather*}
    \phi^{2^s} : V_1^{\otimes2^{s}} \times \dots \times V_{t-s}^{\otimes2^{s}} \rightarrow
    \mathbb{R}
  \end{gather*}
  where $\phi^{2^0} := \phi$ and $\phi^{2^s} := \phi^{2^{s-1}} \ast \phi^{2^{s-1}}$.
\end{definition}

Note that we only define this for exponents which are powers of two because the product $\ast$ is
only defined when the domains of the maps are the same.  An expression like $\phi^3 = \phi \ast
(\phi \ast \phi)$ does not make sense because $\phi$ and $\phi \ast \phi$ have different domains.
This defines the power $\phi^{2^{t-1}}$, which is a linear map $V_1^{\otimes 2^{t-1}} \rightarrow
\mathbb{R}$.

\begin{definition}
  Let $V_1,\dots,V_t$ be finite dimensional vector spaces over $\mathbb{R}$ and let $\phi : V_1
  \times \dots \times V_t \rightarrow \mathbb{R}$ be a $t$-linear map and define $A[\phi^{2^{t-1}}]$
  to be the following square matrix/bilinear map.  Let $u_1,\dots,u_{2^{t-2}},v_1,\dots,v_{2^{t-2}}$
  be vectors where $u_i, v_i \in V_1$.
  \begin{gather*}
    A[\phi^{2^{t-1}}] : V_1^{\otimes 2^{t-2}} \times V_1^{\otimes 2^{t-2}} \rightarrow \mathbb{R} \\
    A[\phi^{2^{t-1}}](u_1 \otimes \dots \otimes u_{2^{t-2}}, v_1 \otimes \dots v_{2^{t-2}})
    \stocnl := \phi^{2^{t-1}}(u_1 \otimes v_1 \otimes u_2 \otimes v_2 \otimes \dots \otimes u_{2^{t-2}}
    \otimes v_{2^{t-2}}).
  \end{gather*}
  Extend the map linearly to the entire domain to produce a bilinear map.
\end{definition}

It is straigtforward to check that by definition, $A[\phi^{2^{t-1}}]$ is a square symmetric
real-valued matrix for any $\phi$; a proof of this fact appears
in~\cite{hqsi-lenz-quasi12-nonregular}.

\subsection{Counting walks and circuits}
\label{sub:c4toeigpowerscountwalks}

This section contains the proof of the following proposition.

\begin{prop} \label{prop:powerscountcycles}
  Let $H$ be a $k$-uniform hypergraph with loops, let $\vec{\pi}$ be a proper ordered partition of
  $k$, and let $\ell \geq 2$ be an integer.  Let $\tau$ be the adjacency map of $H$. Then
  $\tr{A[\tau_{\vec{\pi}}^{2^{t-1}}]^\ell}$ is the number of labeled circuits of type $\vec{\pi}$
  and length $2\ell$ in $H$.
\end{prop}

The proof of this proposition comes down to showing that the function $\tau_{\vec{\pi}}^{2^{t-1}}$
counts the step $S_{\vec{\pi}}$.  We do this by induction by describing exactly the hypergraph
counted by $\tau_{\vec{\pi}}^{2^s}$, which is the following hypergraph.

\begin{definition}
  For $\vec{\pi} = (1,\dots,1)$ with $t$ parts, let $0 \leq s \leq t-1$ and define the hypergraph
  $D_{\vec{\pi},s}$ as follows.  Let $A_1,\dots,A_{t-s}$ be disjoint sets of size $2^{s}$ where
  elements are labeled by binary strings of length $s$ and let $B_{t-s+1},\dots,B_t$ be disjoint
  sets of size $2^{s-1}$ where elements are labeled by binary strings of length $s-1$.  The vertex
  set of $D_{\vec{\pi},s}$ is $A_1 \dot\cup \cdots \dot\cup A_{t-s} \dot\cup B_{t-s+1} \dot\cup
  \cdots \dot\cup B_t$. Make $a_1,\dots,a_{t-s},b_{t-s+1},\dots,b_t$ an edge of $D_{\vec{\pi},s}$ if
  $a_i \in A_i$, $b_j \in B_j$, the codes for $a_1,\dots,a_{t-s}$ are all equal, and the code for
  $b_{t-s+j}$ is equal to the code formed by removing the $j$th bit of the code for $a_1$.

  For a general $\vec{\pi} = (k_1,\dots,k_t)$, start with $D_{(1,\dots,1),s}$ and expand  each
  vertex into the appropriate size; that is, a vertex in $A_i$ is expanded into $k_i$ vertices and
  each vertex in $B_j$ is expanded into $k_j$ vertices.  In $D_{\vec{\pi},s}$, each vertex in $A_i$
  is labeled by a pair $(c,z)$ where $c$ is a bit string of length $s$ and $z \in [k_i]$.  We call
  $z$ the \emph{expansion index} of the vertex.
\end{definition}

The hypergraph $D_{\vec{\pi},0}$ is a single edge and the hypergraph $D_{\vec{\pi},t-1}$ is by
definition the step $S_{\vec{\pi}}$.  The following lemma precisely formulates what we mean when we
say that $\tau_{\vec{\pi}}^{2^s}$ counts the hypergraph $D_{\vec{\pi},s}$.

\begin{lemma} \label{lem:powerscountpartialsteps}
  Let $H$ be a $k$-uniform hypergraph with loops, $\vec{\pi}$ a proper ordered partition of $k$ with
  $\vec{\pi} = (k_1,\dots,k_t)$, and let $0 \leq s \leq t-1$.  Let $W$ be the vector space over
  $\mathbb{R}$ of dimension $|V(H)|$ and let $\tau$ be the adjacency map of $H$.  Let
  $A_1,\dots,A_{t-s},B_{t-s+1},\dots,B_t$ be the vertex sets in the definition of $D_{\vec{\pi},s}$
  and let $\Delta$ be any map $A_1 \cup \dots \cup A_{t-s} \rightarrow V(H)$.  Then
  $\tau_{\vec{\pi}}^s$ counts the number of labeled, possibly degenerate copies of $D_{\vec{\pi},s}$
  extending $\Delta$ as follows.  
  
  Let $a_{i,1},\dots,a_{i,k_i2^s}$ be the vertices of $A_i$ ordered first lexicographically by bit
  code and then for equal codes ordered by expansion index.  Let $\chi_i$ be the indicator tensor in
  $W^{\otimes k_i 2^s}$ for the vertex tuple $(\Delta(a_{i,1}),\dots,\Delta(a_{i,k_i2^s}))$.  Then
  $\tau_{\vec{\pi}}^{2^s}(\chi_1,\dots,\chi_{t-s})$ is the number of edge-preserving maps
  $V(D_{\vec{\pi},s}) \rightarrow V(H)$ which are consistent with $\Delta$.
\end{lemma}

\begin{proof} 
By induction on $s$.  The base case is $s = 0$, where $D_{\vec{\pi},0}$ is a single edge, there are
no $B$-type sets, and thus $\Delta$ is a map $V(D_{\vec{\pi},0}) \rightarrow V(H)$.  The number of
edge preserving maps extending $\Delta$ is either zero or one depending on if the image of $\Delta$
is an edge of $H$ or not.  But $\tau_{\vec{\pi}}(\chi_1,\dots,\chi_t)$ equals zero or one depending
on if the vertices defining the indicator tensors $\chi_i$ form an edge, exactly what is required.

Assume the lemma is true for $s$; we will prove it for $s+1$.  Denote by $\hat{A}_1, \dots,
\hat{A}_{t-s-1}, \hat{B}_{t-s}, \linebreak[1] \dots, \hat{B}_t$ the sets in the definition of
$D_{\vec{\pi},s+1}$ and $A_1,\dots,A_{t-s},B_{t-s+1},\dots,B_t$ the sets in the definition of
$D_{\vec{\pi},s}$.  Let $\hat{\Delta}$ be a map $\hat{A}_1 \cup \dots \cup \hat{A}_{t-s-1}
\rightarrow V(H)$ and let $\hat{\chi_1},\dots,\hat{\chi}_{t-s-1}$ be the indicator tensors for the
image of $\hat{\Delta}$ ordered as in the statement of the lemma.  Since $\hat{\chi}_i$ is an
indicator tensor in $W^{\otimes k_i2^{s+1}}$, it is a simple tensor so $\hat{\chi}_i = \chi_i
\otimes \chi'_i$ for $\chi_i, \chi'_i \in W^{\otimes k_i 2^s}$.  Note that $\chi_i$ is the indicator
tensor for the image under $\hat{\Delta}$ of the vertices of $D_{\vec{\pi},s+1}$ whose code starts
with zero and $\chi'_i$ is the indicator tensor for the image under $\hat{\Delta}$ of the vertices
whose code starts with a one, since the definition of $\hat{\chi}_i$ sorted the vertices in the
image lexicographically.

\begin{figure}[ht] 
\begin{center}
\subfloat[$D_{\vec{\pi},s+1}$]{%
\begin{tikzpicture}[yscale=0.5,rounded corners=10pt,xscale=1.3]
  \foreach \x in {0,1}{%
    \node (a0\x) at (\x,0) [vertex] {};
    \node (a1\x) at (\x,1) [vertex] {};
    \node (a2\x) at (\x,2) [vertex] {};
    \node (a3\x) at (\x,3) [vertex] {};
  }
  \node at (0,-1) {$\hat{A}_1$};
  \node at (1,-1) {$\dots \hat{A}_{t-s-1}$};
  \node at (2.2,-1) {$\hat{B}_{t-s}$};
  \draw (3,0.5) -- (3,2.5) -- (5,2.5) -- (5,0.5) -- cycle;
  \node at (4,1.5) {$B$s};
  \draw (a00) -- (1,0) -- (2,0) -- (3.5,0.9);
  \draw (a10) -- (1,1) -- (2,2.5) -- (3.5,1.3);
  \draw (a20) -- (1,2) -- (2,5) -- (3.5,1.7);
  \draw (a30) -- (1,3) -- (2,7.5) -- (3.5,2.2);

  \node at (2,0) [vertex] {};
  \node at (2,2.5) [vertex] {};
  \node at (2,5) [vertex] {};
  \node at (2,7.5) [vertex] {};

  \begin{scope}[yshift=5cm,dashed]
    \foreach \x in {0,1}{%
      \node (a0\x) at (\x,0) [vertex] {};
      \node (a1\x) at (\x,1) [vertex] {};
      \node (a2\x) at (\x,2) [vertex] {};
      \node (a3\x) at (\x,3) [vertex] {};
    }
    \draw (3,0.5) -- (3,2.5) -- (5,2.5) -- (5,0.5) -- cycle;
    \node at (4,1.5) {$B$s};
    \draw (a00) -- (1,0) -- (2,-5.2) -- (3.5,0.9);
    \draw (a10) -- (1,1) -- (2,-2.7) -- (3.5,1.3);
    \draw (a20) -- (1,2) -- (2,-.1) -- (3.5,1.7);
    \draw (a30) -- (1,3) -- (2,2.5) -- (3.5,2.2);
  \end{scope}
\end{tikzpicture}
} \quad
\subfloat[$D_{\vec{\pi},s}$]{%
\begin{tikzpicture}[rounded corners=10pt,xscale=1.3]
  \foreach \x in {0,1,2}{%
    \node (a0\x) at (\x,0) [vertex] {};
    \node (a1\x) at (\x,1) [vertex] {};
    \node (a2\x) at (\x,2) [vertex] {};
    \node (a3\x) at (\x,3) [vertex] {};
  }
  \node at (0,-1) {$A_1$};
  \node at (1,-1) {$\dots A_{t-s-1}$};
  \node at (2.2,-1) {$A_{t-s}$};
  \draw (3,0.5) -- (3,2.5) -- (5,2.5) -- (5,0.5) -- cycle;
  \node at (4,1.5) {$B$s};
  \draw (a00) -- (a01) -- (2,0) -- (3.5,0.9);
  \draw (a10) -- (a11) -- (2,1) -- (3.5,1.3);
  \draw (a20) -- (a21) -- (2,2) -- (3.5,1.7);
  \draw (a30) -- (a31) -- (2,3) -- (3.5,2.2);
\end{tikzpicture}
}

\begin{align} \label{eq:powerscountwalks}
  \tau_{\vec{\pi}}^{2^{s+1}}(\chi_1 \otimes \chi'_1, \dots, \chi_{t-s-1} \otimes \chi'_{t-s-1}) =
  \sum_{j=1}^{d} \tau_{\vec{\pi}}^{2^s}(\chi_1,\dots,\chi_{t-s-1},w_j)
  \tau_{\vec{\pi}}^{2^s}(\chi'_1,\dots,\chi'_{t-s-1},w_j)
\end{align}

\caption{The induction step of Lemma~\ref{lem:powerscountpartialsteps}}
\label{fig:powerscountwalks}
\end{center}
\end{figure} 

Consider the expansion of the definition of
$\tau_{\vec{\pi}}^{2^{s+1}}(\hat{\chi}_1,\dots,\hat{\chi}_{t-s-1})$ shown in
\eqref{eq:powerscountwalks} in Figure~\ref{fig:powerscountwalks}; the tensors $\hat{\chi}_i$ are
split into $\chi_i$ and $\chi'_i$ and we sum over the standard basis $\{w_1,\dots,w_d\}$ of
$W^{\otimes k_{t-s} 2^s}$, where $d = \dim(W^{\otimes k_{t-s} 2^s})$. We can consider the tensor
$w_j$ in \eqref{eq:powerscountwalks} to be the indicator tensor of a tuple of $k_{t-s}2^s$ vertices.

\begin{definition}
We now describe two embeddings of $D_{\vec{\pi},s}$ into $D_{\vec{\pi},s+1}$.
In Figure~\ref{fig:powerscountwalks}~\textit{(a)}, these two embeddings are the dotted and solid
lines.  Let $\Gamma_0 : V(D_{\vec{\pi},s}) \rightarrow V(D_{\vec{\pi},s+1})$ be the following injection.  For $1
\leq i \leq t -s - 1$ and $a \in A_i$, set $\Gamma_0(a)$ equal to the vertex in $\hat{A}_i$ whose
code equals the code for $a$ with a zero prepended to the code and the same expansion index.  That
is, a vertex in $A_i$ with label $(1011,4)$ is mapped to the vertex in
$\hat{A}_i$ with label $(01011,4)$.  For $a \in A_{t-s}$, set $\Gamma_0(a)$ equal
to the vertex in $\hat{B}_{t-s}$ which has the same label as $a$.  For $t-s+1
\leq j \leq t$ and $b \in B_j$, set $\Gamma_0(b)$ equal to the vertex in $\hat{B}_j$ whose code
equals the code for $b$ with a zero prepended to the code and the same expansion index.  In other
words, $\Gamma_0$ adds a zero to the front of the codes except for vertices in $A_{t-s}$ whose code
does not change.  Define $\Gamma_1 : V(D_{\vec{\pi},s}) \rightarrow V(D_{\vec{\pi},s+1})$ similarly except
prepend a one instead of a zero.  In Figure~\ref{fig:powerscountwalks} \textit{(a)}, the dotted
lines represent $\Gamma_0$ and the solid lines represent $\Gamma_1$.
\end{definition}

\noindent\textit{Claim:} $\Gamma_0$ and $\Gamma_1$ are edge preserving injections and every edge in
$D_{\vec{\pi},s+1}$ is in the image of $\Gamma_0$ or $\Gamma_1$ but not both.

\medskip

\begin{proof}[Proof of Claim] 
Let $E$ be an edge in $D_{\vec{\pi},s}$.  For $1 \leq i \leq j \leq t-s-1$ and $a_i \in A_i \cap E$
and $a_j \in A_j \cap E$, since $E$ is an edge of $D_{\vec{\pi},s}$ the code for $a_i$ equals the
code for $a_j$.  This implies that the codes for $\Gamma_0(a_i)$ and $\Gamma_0(a_j)$ are equal since both
had a zero prepended.  Now consider $b \in A_{t-s} \cap E$ which is mapped to $\hat{B}_{t-s}$.  The
conditions for $\Gamma_0(E)$ an edge of $D_{\vec{\pi},s+1}$ requires that the code for $\Gamma_0(b)$
equals the code formed by deleting the first bit of $\Gamma_0(a)$ where $a \in A_1 \cap E$.  But the
code for $a$ equals the code for $b$ since both are in $A$-type sets in $D_{\vec{\pi},s}$ and the
map $\Gamma_0$ adds a zero to the front of the code for $a$ and leaves the code for $b$ alone.  Thus
the code for $\Gamma_0(b)$ equals the code formed by deleting the first bit of $\Gamma_0(a)$.
Lastly, consider $b \in B_j \cap E$ for $t-s+1 \leq j \leq t$ and consider deleting the $(j+1)$-th
bit of the code for $\Gamma_0(a)$.  This is the same as deleting the $j$-th bit of $a$ since
$\Gamma_0(a)$ had a zero prepended.  But deleting the $j$th bit of $a$ equals the code for $b$,
since $a,b \in E \in E(D_{\vec{\pi},s})$.  Thus deleting the $(j+1)$th bit of $\Gamma_0(a)$ is the
code for $\Gamma_0(b)$.  We have now checked all the conditions, so $\Gamma_0(E)$ is an edge of
$D_{\vec{\pi},s+1}$, i.e.\ $\Gamma_0$ is edge preserving.  $\Gamma_1$ is edge preserving by the same
argument.  Finally, let $E$ be an edge of $D_{\vec{\pi},s+1}$ and pick $a \in E \cap \hat{A}_1$.  If
the first bit of the code for $a$ equals zero, then $E$ is in the image of $\Gamma_0$ and if the
first bit of the code for $a$ equals one, then $E$ is in the image of $\Gamma_1$.  This concludes
the proof of the claim.
\end{proof} 

This claim implies that any edge-preserving map extending $\hat{\Delta}$ is formed from two edge
preserving maps $V(D_{\vec{\pi},s}) \rightarrow V(H)$ each extending the appropriate restriction of
$\hat{\Delta}$.  Start with $\hat{\Delta}$ and extend arbitrarily to a map $\Lambda : \hat{A}_1 \cup
\dots \cup \hat{A}_{t-s-1} \cup \hat{B}_{t-s} \rightarrow V(H)$.  Next define $\Lambda_0$ and
$\Lambda_1$ as maps $A_1 \cup \dots \cup A_{t-s} \rightarrow V(H)$ such that $\Lambda_0 = \Lambda
\circ \Gamma_0|_{\bar{A}}$ and $\Lambda_1 = \Lambda \circ \Gamma_1|_{\bar{A}}$, where $\bar{A} = A_1
\cup \dots \cup A_{t-s}$ so $\Gamma_0|_{\bar{A}}$ is the map $\Gamma_0$ restricted to the $A$-type
sets in $D_{\vec{\pi},s}$.  By the claim, the number of edge-preserving maps extending $\hat{\Delta}$
equals the sum over $\Lambda$ of the product of the number of edge-preserving maps extending
$\Lambda_0$ and extending $\Lambda_1$.  This is because any edge preserving map extending $\Lambda$
can be composed with $\Gamma_0$ and $\Gamma_1$ to create edge preserving maps extending $\Lambda_0$
and $\Lambda_1$, and since $\Gamma_0$ and $\Gamma_1$ are injections covering all edges of
$D_{\vec{\pi},s+1}$, this can be reversed.  The last step in the proof is to show that this is
exactly what \eqref{eq:powerscountwalks} counts.

Let $\hat{b}_1,\dots,\hat{b}_{k_{t-s}2^s}$ be the vertices of $\hat{B}_{t-s}$ listed first in
lexicographic order of codes and then by expansion index.  Let $w$ be the indicator tensor in
$W^{\otimes k_{t-s}2^s}$ for the vertex tuple $(\Lambda(\hat{b}_1), \dots,
\Lambda(\hat{b}_{k_{t-s}2^s}))$.  Note that as $\Lambda$ ranges over all possible extensions of
$\hat{\Delta}$, $w$ ranges over the standard basis of $W^{\otimes k_{t-s}2^s}$.  Now $\chi_1, \dots,
\chi_{t-s-1},w$ are the indicator tensors representing the image of the map $\Lambda_0$, since as
mentioned above, $\chi_1,\dots,\chi_{t-s-1}$ are the indicator tensors for the image under
$\hat{\Delta}$ of the vertices whose code stars with a zero.  Similarly,
$\chi'_1,\dots,\chi'_{t-s-1},w$ are the indicator tensors representing the image of the map
$\Lambda_1$.  Thus by induction, $\tau_{\vec{\pi}}^{2^s}(\chi_1,\dots,\chi_{t-s-1},w)$ is the number
of edge-preserving maps extending $\Lambda_0$ and
$\tau_{\vec{\pi}}^{2^s}(\chi'_1,\dots,\chi'_{t-s-1},w)$ is the number of edge preserving maps
extending $\Lambda_1$.  By the claim, this implies that the product
\begin{align*}
  \tau_{\vec{\pi}}^{2^s}(\chi_1,\dots,\chi_{t-s-1},w)
  \tau_{\vec{\pi}}^{2^s}(\chi'_1,\dots,\chi'_{t-s-1},w)
\end{align*}
counts the number of edge-preserving maps extending $\Lambda$.  Thus
\eqref{eq:powerscountwalks} sums over the choices for $\Lambda$ extending $\hat{\Delta}$ of
the number of edge-preserving maps extending $\Lambda$.  This sum is exactly the number of
edge-preserving maps extending $\hat{\Delta}$, so the proof is complete.
\end{proof} 

\begin{cor} \label{cor:powerscountwalks}
  Let $H$ be a $k$-uniform hypergraph with loops, $\vec{\pi}$ a proper ordered partition of $k$ with
  $\vec{\pi} = (k_1,\dots,k_t)$, and let $\ell \geq 2$ be an integer.  Let $W$ be the vector space
  over $\mathbb{R}$ of dimension $|V(H)|$ and let $\tau$ be the adjacency map of $H$.  Let
  $a_1,\dots,a_{k_1 2^{t-2}},a'_1,\dots,a'_{k_1 2^{t-2}}$ be (not necessarily distinct) vertices of
  $H$ and let $\xi$ and $\xi'$ be the indicator tensors in $W^{k_12^{t-2}}$ for the tuples
  $(a_1,\dots,a_{k_12^{t-2}})$ and $(a'_1,\dots,a'_{k_12^{t-2}})$ respectively.  Then
  $A[\tau_{\vec{\pi}}^{2^{t-1}}](\xi,\xi')$ is the number of labeled, possibly degenerate steps of
  type $\vec{\pi}$ in $H$ with attach tuples $(a_1,\dots,a_{k_12^{t-2}})$ and
  $(a'_1,\dots,a'_{k_12^{t-2}})$.  Also, $A[\tau_{\vec{\pi}}^{2^{t-1}}]^{\ell}(\xi,\xi')$ is the
  number of labeled walks of length $2\ell$ and type $\vec{\pi}$ with attach tuples
  $(a_1,\dots,a_{k_12^{t-2}})$ and $(a'_1,\dots,a'_{k_12^{t-2}})$.
\end{cor}

\begin{proof} 
The proof is by induction on $\ell$.  First, consider the base case of $\ell = 1$, where the path of
length two and type ${\vec{\pi}}$ is the step of type ${\vec{\pi}}$.  Let $A$ be the vertex set from
the definition of the step $S_{\vec{\pi}}$.  Define a mapping $\Delta : A \rightarrow V(H)$ by
mapping the attach tuples of $S_{\vec{\pi}}$ to the tuples $(a_1,\dots,a_{k_12^{t-2}})$ and
$(a'_1,\dots,a'_{k_12^{t-2}})$ in $V(H)$.  By definition, the first attach tuple of $S_{\vec{\pi}}$
is the vertices ending with a zero and listed in lexicographic order and the second attach tuple of
$S_{\vec{\pi}}$ is the vertices ending with a one and listed in lexicographic order.  This implies
that the indicator tensor $\chi_1$ from the statement of Lemma~\ref{lem:powerscountpartialsteps} is
the indicator tensor in $W^{\otimes k_12^{t-1}}$ for the tuple $(a_1, \dots, a_{k_1}, a'_1, \dots,
a'_{k_1}, a_{k_1+1}, \dots, a_{2k_1}, a'_{k_1+1}, \dots, a'_{2k_1}, \dots, a_{k_12^{t-3}+ 1}, \dots,
a_{k_12^{t-2}}, \linebreak[1] a'_{k_12^{t-3}+1}, \dots, a'_{k_12^{t-2}})$, since each attach tuple
is in lexicographic order but the last bit is zero or one so the full ordering alternates between
attach tuples.  By the definition of $A[\tau_{\vec{\pi}}^{2^{t-1}}]$ and the indicator tensors
$\xi,\xi',\chi_1$, $A[\tau_{\vec{\pi}}^{2^{t-1}}](\xi,\xi') = \tau_{\vec{\pi}}^{2^{t-1}}(\chi_1)$.
Thus Lemma~\ref{lem:powerscountpartialsteps} applied with $s = t-1$ shows that the number of
edge-preserving maps extending $\Delta$ is $A[\tau_{\vec{\pi}}^{2^{t-1}}](\xi,\xi')$, but by the
definition of $\Delta$, this is exactly the number of labeled, possibly degenerate steps of type
${\vec{\pi}}$ with attach tuples $(a_1,\dots,a_{k_12^{t-2}})$ and $(a'_1,\dots,a'_{k_12^{t-2}})$.

Next assume that the corollary is true for $\ell$; we will show that it is true for $\ell + 1$.
Using the definition of matrix multiplication, let $\{d_1,\dots,d_{\dim(W^{\otimes k_1 2^{t-2}})}
\}$ be the standard basis of $W^{\otimes k_1 2^{t-2}}$ so
\begin{align} \label{eq:powerofAincountingcycles}
  A[\tau_{\vec{\pi}}^{2^{t-1}}]^{\ell+1}(\xi,\xi') = \sum_{i=1}^{\dim(W^{\otimes k_1 2^{t-2}})}
  A[\tau_{\vec{\pi}}^{2^{t-1}}]^\ell(\xi,d_i) A[\tau_{\vec{\pi}}^{2^{t-1}}](d_i,\xi').
\end{align}
Each standard basis vector $d_i$ can be thought of as a $k_1 2^{t-2}$-tuple of vertices which
corresponds to one of the two attach tuples.  Thus \eqref{eq:powerofAincountingcycles} sums over the
internal attach tuple for a walk of length $2\ell$ and $S_{\vec{\pi}}$.
\end{proof} 

\begin{proof}[Proof of Proposition~\ref{prop:powerscountcycles}] 
Since $A[\tau_{\vec{\pi}}^{2^{t-1}}]^\ell$ counts the number of walks of length $2\ell$, the trace
counts circuits.  If $\{d_1,\dots,d_{\dim(W^{\otimes k_1 2^{t-2}})} \}$ is any orthonormal basis of
$W^{\otimes k_1 2^{t-2}}$, the trace of the matrix $A[\tau_{\vec{\pi}}^{2^{t-1}}]^\ell$ is
\begin{align*}
  \tr{A[\tau_{\vec{\pi}}^{2^{t-1}}]^\ell} = \sum_{i=1}^{\dim(W^{\otimes k_1 2^{t-2}})}
  A[\tau_{\vec{\pi}}^{2^{t-1}}]^\ell(d_i, d_i).
\end{align*}
If $\{d_1,\dots,d_{\dim(W^{\otimes k_1 2^{t-2}})} \}$ is the standard basis, each $d_i$ corresponds
to a tuple of $k_12^{t-2}$ vertices, so the above expression is the number of walks of type
${\vec{\pi}}$ with both attach tuples equal to $d_i$.
\end{proof} 

\subsection{Bounding eigenvalues from cycle counts}
\label{sub:c4toeigfinalproof}

This section contains the proof that \propcycle[4\ell]{$\pi$} $\Rightarrow$ \propeig{$\pi$} for
$d$-coregular hypergraphs with loops.  First, we require a few simple algebraic facts of multilinear
maps.

\begin{lemma} \label{lem:singlepowerupperbound}
  Let $t \geq 2$, let $V_1, \dots, V_t$ be finite dimensional vector spaces over $\mathbb{R}$, let
  $\phi : V_1 \times \dots \times V_t \rightarrow \mathbb{R}$ be a $t$-linear map and let $x_1 \in
  V_1,\dots,x_t \in V_t$ be unit length vectors. Then
  \begin{align*}
    \left| \phi(x_1,\dots,x_t) \right|^2 \leq \left| \phi^2(x_1 \otimes x_1, \dots, x_{t-1}
    \otimes x_{t-1}) \right|.
  \end{align*}
\end{lemma}

\begin{proof} 
Consider the linear map $\phi(x_1, \dots, x_{t-1},\cdot)$ which is a linear map from $V_{t}$ to
$\mathbb{R}$.  There exists a vector $w \in V_t$ such that
$\phi(x_1,\dots,x_{t-1},\cdot) = \dotp{w}{\cdot}$. Then
\begin{align*}
  \phi^2(x_1 \otimes x_1, \dots, x_{t-1} \otimes x_{t-1})
  = \sum_j \left| \phi(x_1, \dots, x_{t-1}, b_j)\right|^2 
  = \sum_j \left| \dotp{w}{b_j} \right|^2 = \dotp{w}{w}
\end{align*}
where the last equality is because $\{b_j\}$ is an orthonormal basis of $V_{t}$. Since $\left\lVert
w \right\rVert = \sqrt{\dotp{w}{w}}$, $\left| \phi^2(x_1 \otimes x_1, \dots, x_{t-1} \otimes
x_{t-1})\right| = \left| \dotp{w}{w} \right| = \left| \dotp{w}{w/\left\lVert w \right\rVert}
\right|^2$.  But since $x_t$ is unit length and $\dotp{w}{\cdot}$ is maximized over the unit ball at
vectors parallel to $w$ (so maximized at $\pm w/\left\lVert w \right\rVert$), $\left|
\dotp{w}{w/\left\lVert w \right\rVert} \right| \geq \left| \dotp{w}{x_t} \right|$. Thus
\begin{align*}
  \left| \phi^2(x_1 \otimes x_1, \dots, x_{t-1} \otimes x_{t-1})\right| = \left|
  \dotp{w}{\frac{w}{\left\lVert w \right\rVert}} \right|^2 \geq \left| \dotp{w}{x_t} \right|^2
  = \left| \phi(x_1,\dots,x_t) \right|^2.
\end{align*}
The last equality used the definition of $w$, that $\phi(x_1,\dots,x_{t-1},\cdot) =
\dotp{w}{\cdot}$.
\end{proof} 

\begin{lemma} \label{lem:largesteigpower}
  Let $t \geq 2$, let $V_1,\dots,V_t$ be finite dimensional vector spaces over $\mathbb{R}$, and let
  $\phi : V_1 \times \dots \times V_t \rightarrow \mathbb{R}$ be a $t$-linear map. Then for any unit
  length $x_1 \in V_1, \dots, x_t \in V_t$, we have
  \begin{align} \label{eq:largesteigpower}
    \left| \phi(x_1,\dots,x_t) \right|^{2^{t-1}} \leq \left| A[\phi^{2^{t-1}}](
    \underbrace{x_1 \otimes \dots \otimes x_1}_{2^{t-2}},
    \underbrace{x_1 \otimes \dots \otimes x_1}_{2^{t-2}}
    ) \right|.
  \end{align}
  Also,
  \begin{align} \label{eq:lambdaoneboundsspectralnorm}
    \left\lVert \phi \right\rVert^{2^{t-1}} \leq \lambda_1(A[\phi^{2^{t-1}}]).
  \end{align}
\end{lemma}

\begin{proof} 
By induction on $s$ we have that 
\begin{align*}
    \left| \phi(x_1,\dots,x_t) \right|^{2^s} \leq \left| \phi^{2^s}(
    \underbrace{x_1 \otimes \dots \otimes x_1}_{2^{s}}, \dots,
    \underbrace{x_{t-s} \otimes \dots \otimes x_{t-s}}_{2^{s}}
    ) \right|
\end{align*}
Indeed, the base case is $s = 0$ where both sides are equal and the induction step follows from the
previous lemma, since
\begin{align*}
  \Big(\left| \phi(x_1,\dots,x_t) \right|^{2^{s-1}}\Big)^2 &\leq \left| \phi^{2^{s-1}}(
    \underbrace{x_1 \otimes \dots \otimes x_1}_{2^{s-1}}, \dots,
    \underbrace{x_{t-s+1} \otimes \dots \otimes x_{t-s+1}}_{2^{s-1}}
    ) \right|^2 \\
    &\leq \left| \phi^{2^s}(
    \underbrace{x_1 \otimes \dots \otimes x_1}_{2^{s}}, \dots,
    \underbrace{x_{t-s} \otimes \dots \otimes x_{t-s}}_{2^{s}}
    ) \right|.
\end{align*}
By definition of $A[\phi^{2^{t-1}}]$, $| A[\phi^{2^{t-1}}]( x_1 \otimes \dots
\otimes x_1, x_1 \otimes \dots \otimes x_1) | = | \phi^{2^{t-1}}( x_1 \otimes \dots \otimes x_1) |$,
completing the proof of \eqref{eq:largesteigpower}.  Let $x_1, \dots, x_t$ as unit length
vectors maximizing $\phi$. Since $x_1 \otimes \dots \otimes x_1$ is unit length,
\eqref{eq:largesteigpower} proves that 
\begin{align*}
  \left\lVert \phi \right\rVert^{2^{t-1}} = \left| \phi(x_1,\dots,x_t) \right|^{2^{t-1}} \leq
  \left|A[\phi^{2^{t-1}}](x_1\otimes\dots\otimes x_1, x_1\otimes\dots\otimes x_1) \right|
  \leq \lambda_1(A[\phi^{2^{t-1}}]).
\end{align*}
\end{proof} 

\begin{cor} \label{cor:secondeig}
  Let $H$ be a $d$-coregular, $k$-uniform hypergraph with loops and let $\pi$ be any proper
  partition of $k$ with $t$ parts.  Then for any ordering $\vec{\pi}$ of $\pi$,
  \begin{align*}
    \lambda_{2,\pi}(H) \leq \Big( \lambda_2(A[\tau_{\vec{\pi}}^{2^{t-1}}]) \Big)^{2^{-t+1}}.
  \end{align*}
\end{cor}

\begin{proof} 
Let $\vec{\pi} = (k_1,\dots,k_t)$ and let $x_1, \dots, x_t$ be unit length vectors maximizing
$\tau_{\vec{\pi}} - \frac{d}{n} J_{\vec{\pi}}$ in absolute value, so that $\lambda_{2,\pi}(H) =
\left| (\tau_{\vec{\pi}} - \frac{d}{n}J_{\vec{\pi}})(x_1,\dots,x_t) \right|$.  Write $x_1 = \alpha y
+ \beta \hat{1}$, where $y$ is a unit length vector perpendicular to the all-ones vector, $\hat{1}$
is the all-ones vector scaled to unit length, and $\alpha, \beta \in \mathbb{R}$ with $\alpha^2 +
\beta^2 = 1$.  Let $W$ be the vector space over $\mathbb{R}$ of dimension $n$ and for $1 \leq i \leq
t$ let $e_{i,1}, \dots, e_{i,n^{k_i}}$ be the standard basis of $W^{\otimes k_i}$.  Since $H$ is
$d$-coregular,
\begin{align*}
  \tau_{\vec{\pi}}(\hat{1},x_2,\dots,x_t) &=
  \frac{1}{n^{k/2}} 
  \sum_{1 \leq j_2 \leq n^{k_2}} \cdots \sum_{1 \leq j_t \leq n^{k_t}}
  \dotp{e_{2,j_2}}{x_2} \cdots \dotp{e_{t,j_t}}{x_t}
  \sum_{1 \leq j_1 \leq n^{k_1}} \tau_{\vec{\pi}}(e_{1,j_1}, \dots, e_{t,j_t})
  \\ &=
  \frac{1}{n^{k/2}} 
  \sum_{1 \leq j_2 \leq n^{k_2}} \cdots \sum_{1 \leq j_t \leq n^{k_t}}
  \dotp{e_{2,j_2}}{x_2} \cdots \dotp{e_{t,j_t}}{x_t}
  d n^{k_1-1} J_{\vec{\pi}}(e_{1,1}, e_{2,j_2}, \dots, e_{t,j_t})
  \\ &=
  \frac{d}{n} J_{\vec{\pi}}(\hat{1}, x_2, \dots, x_t).
\end{align*}
Next, $J_{\vec{\pi}}(y,x_2,\dots,x_t) = \dotp{1}{y} \dotp{1}{x_2} \cdots \dotp{1}{x_t}$.  Since $y$
is perpendicular to the all-ones vector, $J_{\vec{\pi}}(y,x_2,\dots,x_t) = 0$.  Therefore, using
linearity,
\begin{align*}
  \lambda_{2,\pi}(H) = \left| \left(\tau_{\vec{\pi}} - \frac{d}{n}J_{\vec{\pi}}\right)(\alpha y + \beta \hat{1},
  x_2,\dots,x_t) \right| = \left| \alpha \right| \left| \tau_{\vec{\pi}}(y, x_2,\dots,x_t) \right|.
\end{align*}
By \eqref{eq:largesteigpower} applied to $\tau_{\vec{\pi}}(y,x_2,\dots,x_t)$, $\lambda_{2,\pi}(H)
\leq \left| \alpha \right| | A[\tau_{\vec{\pi}}^{2^{t-1}}](y\otimes \dots \otimes y, y\otimes \dots
\otimes y)|^{2^{-t+1}}$.  Since $H$ is $d$-coregular, the number of steps of type $\vec{\pi}$ with a
fixed attach tuple $A^{(0)}$ is independent of the choice of $A^{(0)}$.  By
Corollary~\ref{cor:powerscountwalks}, each row of the matrix $A[\tau_{\vec{\pi}}^{2^{t-1}}]$
cooresponds to an attach tuple $A^{(0)}$ and the sum of the entries in that row counts the number of
steps of type $\vec{\pi}$ with fixed attach tuple $A^{(0)}$.  Therefore, each row of
$A[\tau_{\vec{\pi}}^{2^{t-1}}]$ sums to the same value so that the Perron-Frobenius Theorem implies
that the all-ones vector is the eigenvector associated to
$\lambda_1(A[\tau_{\vec{\pi}}^{2^{t-1}}])$.  Since $y \otimes \dots \otimes y$ is perpendicular to
the all-ones vector and $A[\tau_{\vec{\pi}}^{2^{t-1}}]$ is a square real symmetric matrix,
$|A[\tau_{\vec{\pi}}^{2^{t-1}}](y\otimes\dots\otimes y, y\otimes \dots \otimes y)| \leq
\lambda_2(A[\tau_{\vec{\pi}}^{2^{t-1}})$.  Since $\left| \alpha \right| \leq 1$, the proof is
complete.
\end{proof} 

\begin{proof}[\prooftext that {\texttt{Cycle}$_{4\ell}$[$\pi$]} $\Rightarrow$ \propeig{$\pi$}] 
\label{proof:c4toeig}
Let $\mathcal{H} = \{H_n\}_{n \rightarrow \infty}$ be a sequence of $d$-coregular, $k$-uniform
hypergraphs with loops and let $\tau_n$ be the adjacency map of $H_n$.  For notational convenience,
the subscript on $n$ is dropped below.  Let $\vec{\pi}$ be any ordering of the entries of $\pi$.
Let $m = |E(C_{\pi,4\ell})| = 2\ell 2^{t-1}$ and note that $|V(C_{\pi,4\ell})| = mk/2$ since
$C_{\pi,4\ell}$ is two-regular.  The matrix $A = A[\tau_{\vec{\pi}}^{2^{t-1}}]$ is a square
symmetric real valued matrix, so let $\mu_1,\dots, \mu_r$ be the eigenvalues of $A$ arranged so that
$|\mu_1| \geq \dots \geq |\mu_r|$, where $r = \dim(A)$.  The eigenvalues of $A^{2\ell}$ are
$\mu_1^{2\ell}, \dots, \mu_d^{2\ell}$ and the trace of $A^{2\ell}$ is $\sum_i \mu_i^{2\ell}$.  Since
all $\mu_i^{2\ell} \geq 0$, Proposition~\ref{prop:powerscountcycles} and \propcycle[4\ell]{$\pi$}
imply that
\begin{align} \label{eq:c4toeigmain}
  \mu_1^{2\ell} \leq \mu_1^{2\ell} + \mu_2^{2\ell} 
  &\leq \tr{A^{2\ell}}
  \stocnla = \#\{\text{possibly degen } C_{\pi,4\ell} \,\, \text{in } H_n\} 
  \stocnla \leq p^m n^{mk/2} + o(n^{mk/2}).
\end{align}
Since $pn^{k/2} = dn^{k/2-1} = \tau_{\vec{\pi}}(\hat{1},\dots,\hat{1}) \leq \left\lVert
\tau_{\vec{\pi}} \right\rVert = \lambda_{1,\pi}(H)$, \eqref{eq:lambdaoneboundsspectralnorm} implies
that $\mu_1 \geq p^{2^{t-1}}n^{k2^{t-2}}$ which implies equality up to $o(n^{mk/2})$ throughout
\eqref{eq:c4toeigmain}.  Therefore, $\mu_2 = o(n^{k2^{t-2}})$ so that Corollary~\ref{cor:secondeig}
shows that $\lambda_{2,\pi}(H) = o(n^{k/2})$, completing the proof.
\end{proof} 

\medskip \emph{Acknowledgements.} The authors would like to thank Vojt\u{e}ch R\"odl, Mathias Schacht,
and the referees for helpful discussion and feedback.

\bibliographystyle{abbrv}
\bibliography{refs.bib}


\end{document}